\documentclass[12pt,a4paper]{article}
\pdfoutput=1
\usepackage{jheppub}
\usepackage{tikz} 

 \usepackage{amsmath, amssymb, latexsym, amsthm, amsfonts, mathrsfs,
   amscd, mathtools, etoolbox, pgfkeys, pgfopts, xparse, xstring}
 \usepackage{xcolor}  
 \usepackage{caption} 
 \usepackage{subcaption}
 \usepackage{slashed}   
\usepackage{float}      
\usepackage{dynkin-diagrams}  
\usetikzlibrary{chains}
   
\usepackage{graphicx} 
\graphicspath{ {project2 tex/} }

\def\CA{{\mathcal{A}}}

\def\CN{{\mathcal{N}}}

\def\CR{{\mathcal{R}}}

\newcommand{\be} {\begin{equation}}
\newcommand{\ee} {\end{equation}}
\newcommand{\bea} {\begin{eqnarray}}
\newcommand{\eea} {\end{eqnarray}}
\newcommand{\ba} {\begin{array}}
\newcommand{\ea} {\end{array}}
\newcommand{\nn} {\nonumber}

\theoremstyle{plain} 
\newtheorem{thm}{Theorem}[section]

\newtheorem{cor}[thm]{Corollary}
\newtheorem{prop}[thm]{Proposition}
 
\theoremstyle{definition}

\newcommand{\sgn} { {\rm sgn}}

 \title{Iterated integrals for generalized Appell functions}
 \author{Aradhita Chattopadhyaya$^a$, Jan Manschot$^{b,c}$}
\affiliation{ $^a$Chennai Mathematical Institute,\\
H1 SIPCOT IT Park\\
Siruseri, Kelambakkam
Tamil Nadu 603103, India\\
$^b$School of Mathematics, Trinity College, Dublin 2,
   Ireland\\
   $^c$Hamilton Mathematical Institute, Trinity College, Dublin 2,
   Ireland\\}

\abstract{Appell functions are a large class of multi-variable
  quasi-elliptic functions, which are also instances of higher depth mock
  modular forms. A central aspect of these functions is their non-holomorphic
modular completion. Our previous work developed Appell functions for a
general positive definite lattice $\Lambda$, and expressed the
non-holomorphic completion in terms of the generalized error functions
$M_P$, which are integrals over a $P$-dimensional hyperplane in
$\Lambda\otimes \mathbb{C}$. In this sequel, we express $M_P$ as a $P$-dimensional iterated integral of
variables $w_j\in \mathbb{H}$, such that the completion of the Appell
function takes the form of an iterated Eichler integral of modular forms. We apply this to the case of the root lattice of the
$A_N$ Lie algebra, which is of particular interest because of their
appearance in partition functions of topological twisted Yang-Mills theories.
}    
\begin{document} 
\maketitle       

\section{Introduction}  
\label{sec:intro}

Appell functions are multi-variable, quasi-elliptic functions going
back to the late 19'th century \cite{Appell:1886, Lerch:1892}. Recently, they have been recognized as important
examples of mock modular forms \cite{ZwegersThesis}, which themselves
have a long history going back to Ramanujan (1887-1920). He
wrote in his last letter to Hardy about 17 $q$-series which he called
`mock theta' functions. Precisely these
functions appear in theoretical
physics and mathematics, for example as partition functions of certain
three-dimensional ${\cal N}=2$ 
supersymmetric theories \cite{Cheng:2018vpl, Cheng:2022rqr}, and as
characters in the context of umbral moonshine \cite{Cheng:2012tq} . 

The mock theta functions have led to the more general concept of ``mock
modular forms'' \cite{MR2605321, ZwegersThesis}, and it is
established that all Ramanujan's 17 $q$-series are instances of these functions
\cite{ramanujan_lost_notebook, ZwegersThesis}.
Mock modular forms appear in many places in theoretical
physics, in particular as statistical partition functions in
supersymmetric theories, such as compactifications of Type II 
string theory on a Calabi-Yau three-fold $X$ to ${\cal
  N}=2$ supergravity, and in the specific case of $X={\rm K3}\times T^2$ to ${\cal N}=4$ supergravity. These compactifications give rise
to BPS states, which preserve part of the supersymmetry. Mock modular
forms arise as generating functions of degeneracies of BPS states, which are widely studied
for their applications to the entropy of BPS black holes 
\cite{Manschot:2007ha, Manschot:2009ia,
  Dabholkar:2012nd, Alexandrov:2012au, KumarGupta:2018rac, Alexandrov:2019rth,
  Chattopadhyaya:2021rdi, Alexandrov:2023ltz, Alexandrov:2024jnu,
  Alexandrov:2024wla, Pioline:2025xgf, Alexandrov:2025sig}. For ${\cal
  N}=4$ string compactifications on 
$K3\times T^2$ and CHL orbifolds, the generating function of  
Hurwitz-Kronecker class numbers appears as an element of black hole partition functions
\cite{Dabholkar:2012nd, Chattopadhyaya:2018xvg}. This generating function of class numbers is
a famous example of a mock modular form \cite{MR2605321},  whose properties were established by
Zagier \cite{zagier:75, HirzebruchZagier1976}. 
     
 This function also appears as the partition function of Vafa-Witten
 theory with gauge group $SU(2)$ on $\mathbb{CP}^2$ \cite{Klyachko:91,
   Vafa:1994tf}. This theory is a topological twisted version of $\CN=4$
 supersymmetric Yang-Mills theory, whose partition function on
 $\mathbb{CP}^2$ is a generating function of Euler numbers of moduli
 spaces of instantons, or within algebraic geometry, moduli spaces of rank 2, semi-stable coherent sheaves. This correspondence can be understood following the
 Donaldson-Uhlenbeck-Yau theorem \cite{Donaldson:1985zz,Uhlenbeck:1986ntc}. Moreover, the 
partition functions can be refined with an extra variable, to obtain generating functions of Poincar\'e polynomials
of moduli spaces of semi-stables sheaves
\cite{Yoshioka94,Yoshioka:k95}. Remarkably, these can be expressed as a specialization of the
Appell-Lerch function \cite{Bringmann:2010sd}, whose modular
properties were analyzed in \cite{semikhatov2005, ZwegersThesis}.

For gauge groups $SU(N)$, $N>2$, the partition
functions can be expressed in terms of generalizations of Appell
functions $\Phi_{\mu,\nu}$ as in Eq. \eqref{eq:DefPhimunu} below
\cite{Manschot:2011dj, Manschot:2011ym, Manschot:2014cca}. These are 
examples of higher depth mock modular forms, which have intriguing mathematical
properties \cite{Alexandrov:2016enp, Nazaroglu:2016lmr, Bringmann:2018cov, Bringmann:2026wea}. The structure of
non-holomorphic terms satisfies an elegant holomorphic anomaly 
equation, which can be understood from the physical path integral
\cite{Vafa:1994tf, Minahan:1998vr, Manschot:2017xcr, 
  Alexandrov:2019rth, Dabholkar:2020fde, Manschot:2021qqe}. The
non-holomorphic terms correspond to Abelian instantons for $U(1)$
subgroups of the gauge group $G$. While these do not saturate the BPS
bound, they do contribute to the path integral for four-manifolds with
$b_2^+=1$ \cite{Vafa:1994tf, Moore:1997pc}, which includes $\mathbb{CP}^2$. Therefore, the
partition functions for gauge group $G$ involve non-holomorphic terms involving theta
series for the (dual of the) root lattice of $G$ and its sublattices.

In the recent work \cite{Chattopadhyaya:2025bkq}, we determined the
modular completion of Appell functions $\widehat \Phi_{\mu,\nu}$. To this end, the
denominators in Eq. \eqref{eq:DefPhimunu} below are expanded using geometric sums, such
that the techniques for indefinite theta series can be used
\cite{Vigneras:1977, Alexandrov:2016enp, Nazaroglu:2016lmr,
  FunkeNotes:2018}. In particular, $\widehat \Phi_{\mu\nu}$ is
determined through their connection to indefinite theta series, and
then making appropriate replacements in the kernel in terms of
functions $E_L$ and $M_L$ of order $L$. These functions can be
 thought of as higher dimensional
 generalizations of the error function $E_1(x)={\rm Erf}(\sqrt{\pi}x)=
 2\int_{0}^{x}e^{-\pi t^2}dt$ and complementary error function ${\rm
   Erfc}$, which satisfy Vign\'eras' theorem
 for kernels of theta series \cite{Vigneras:1977}. In this way, the mock modular
 completion of the rank $N$ Vafa-Witten instanton partition functions
 is given in terms of these generalized error functions $M_L$ and Appell
 functions of depth $N-L$ where $L$ is an integer and $0<L<N$. This
 approach has been applied to other four-manifolds with $b_2^+=1$ in
 Ref. \cite{Alexandrov:2026lrd}. 
 
 In this sequel to \cite{Chattopadhyaya:2025bkq}, we express the
 generalized error functions $M_L$ appearing in $\widehat \Phi_{\mu,\nu}$ as
 iterated integrals in Theorem \ref{theorem}. This extends previous results for $N=1,2$
 \cite{Alexandrov:2016enp, bringmann2018higher, Manschot:2017xcr}.
 Including the lattice sums for $\widehat \Phi_{\mu,\nu}$, these integrals give rise to iterated Eichler
 integrals, which are iterated integrals whose integrands are products of
 modular forms. This elegant class of 
 integrals makes the modular properties manifest, and are also 
 natural from the perspective of physical path integrals
 \cite{Manschot:2017xcr, Alexandrov:2019rth, Dabholkar:2020fde}.  

The paper is structured as follows. 
Section \ref{review} reviews our previous work
\cite{Chattopadhyaya:2025bkq} on Appell functions and their relation
to mock modularity, and sets the notation for this paper. Using
results of the later sections, it also derives an iterative form for
the modular completion $\widehat \Phi_{\mu,\nu}$. Section
\ref{sec:error} gives a derivation of the generalized error
functions $M_N$ appearing in the modular completion of these Appell
functions as an $N$-dimensional iterated Eichler integral. Section
\ref{ANlattice} applies these results to the $A_N$ root lattice. 
Section \ref{ortho} explores the structure of the function
$E_N$ appearing in the completion of the Appell function for an $A_N$
lattice, with Section \ref{A9} discussing a few examples for $A_9$. Finally, Appendix
\ref{reviewVig} recalls the statement of 
Vign\'eras' theorem \cite{Vigneras:1977} for an even lattice. This simplification is
relevant to the application to $A_N$ lattices. For the $A_N$ lattice, Appendix \ref{app:ANlattice} provides a more explicit rederivation of
Theorem \ref{theorem} for $M_N$.

\subsection*{Acknowledgments}
We thank Sergey Alexandrov for discussions and comments on a draft. The work of A.C. is funded by Ramanujan Fellowship, RJF/2023/000070 offered by Anusandhan National Research Foundation (ANRF), India.

\section{Appell functions and mock modularity}\label{review}

\subsection{The classical Appell function}
The classical Appell function is a two-variable, quasi-elliptic
function, $\varphi:\mathbb{H}\times \mathbb{C}\times \mathbb{C}\rightarrow \mathbb{C}$,
\begin{equation}
  \begin{split}
	\varphi(\tau,u,v)&= e^{\pi i u}\sum_{n\in\mathbb{Z}}
        \frac{(-1)^nq^{n(n+1)/2}e^{2\pi i vn}}{1-e^{2\pi
            iu}q^n}.
\end{split}
\end{equation}
Zwegers' PhD thesis \cite{ZwegersThesis} studies the closely related Appell-Lerch sum
$\mu(\tau,u,v)=\varphi(\tau,u,v)/\vartheta_1(\tau,v)$, with $\vartheta_1$ the Jacobi theta series
  \be
\vartheta_1(\tau,v)=i\sum_{n\in \mathbb{Z}+\frac{1}{2}}
(-1)^{n-\frac{1}{2}} q^{\frac{n^2}{2}}e^{2\pi i n v}.
\ee
He describes the modular properties of this
function extensively using the Mordell
integral. An alternative method to derive these properties is to relate the
Appell function to an indefinite theta series. To this end, one expands
the denominator as a geometric series,
\begin{equation}
  \begin{split}
	\varphi(\tau,u,v)&=\frac{1}{2}\sum_{\begin{smallmatrix}n\in\mathbb{Z}\\
            s\in\mathbb{Z}\end{smallmatrix}}(-1)^nq^{n(n+1)/2+ns}e^{2\pi i
          (vn+us)}(\sgn(s+\epsilon)+\sgn(n+a)),
        \end{split}
\end{equation}
where $a={\rm Im}(u)/{\rm Im}(\tau)$ and $\epsilon$ satisfies
$0<\epsilon\ll 1$. This function does not transform as a modular form,
but a related non-holomorphic function $\widehat \varphi$ does. The latter is essentially
obtained by replacing $\sgn(n+a)$ by $E_1((n+a)\sqrt{2\tau_2})$ with
$\tau_2={\rm Im}(\tau)$, such that
the kernel satisfies the conditions of Vign\'eras' theorem
\cite{Vigneras:1977}. The theorem is reviewed in Appendix
\ref{reviewVig} for an even lattice. Note that for $\tau_2\to \infty$,
$E_1$ approaches $\sgn$, such that the contribution is subleading to
the holomorphic part.

In this way, one derives that completion $\widehat \varphi$ \cite{ZwegersThesis}, 
\be
\label{eq:whphi}
\widehat \varphi(\tau,\bar \tau,u,\bar u, v,\bar v)=\varphi(\tau,u,v)+\frac{i}{2}R(\tau,\bar
\tau, u-v,\bar u -\bar v)\, \vartheta_1(\tau,v),
\ee
with
\be
\begin{split}
&R(\tau,\bar \tau,u,\bar u)\\
&\qquad =\sum_{n \in
    \mathbb{Z}+1/2}
  (\sgn(n)-E_1((n+a)\sqrt{2\tau_2}))\,(-1)^{n-1/2}q^{-n^2/2}e^{-2\pi
    i n u}.
\end{split}
  \ee

\subsection{Modular completion of Appell function}\label{lastpaperrev}

In our previous work \cite{Chattopadhyaya:2025bkq}, we apply the 
approach explained for $\varphi$ to generalized Appell functions $\Phi_{\mu,\nu}$ and obtain their modular
completion $\widehat \Phi_{\mu,\nu}$. Following previous works
\cite{Alexandrov:2016enp, Nazaroglu:2016lmr}, these are expressed in
terms of generalized error functions $E_N$ and $M_N$. In this section, we briefly review these mock modular forms and set the notation for this paper.

We start by introducing the Appell functions for a general lattice. 
Let $\Lambda$ be an $N$-dimensional lattice with quadratic form
$Q$ and bilinear form $B$. Further, let $\{d_r\}$ be a set of $M\leq
N$ linearly independent vectors $d_r\in \Lambda$, which span the
sublattice $\Lambda_d\subseteq \Lambda$. We then introduce the Appell
function $ \Phi_{\mu,\nu}$,
\begin{eqnarray}
\label{eq:DefPhimunu}
  \Phi_{\mu,\nu}(\tau,u,v,\{ d_r \}) &= & e^{2\pi i B(\nu,u)}\sum_{k\in{\Lambda+{\mu-\nu}}}\frac{q^{Q(k)/2+B(\nu,k)}e^{2\pi iB({v},k)}}{\prod_{r=1}^M(1-e^{2\pi i B(d_r,u)} q^{B(d_r,k)})}.
\end{eqnarray}
where $u\in\Lambda_d\otimes\mathbb{C}$, $v\in \Lambda \otimes
\mathbb{C}$, $\mu\in \Lambda^*/\Lambda$ and $\nu\in \Lambda_d^*$. Note that
$\Phi_{\mu,\nu}$ is invariant under $\mu\mapsto \mu + \ell$ with $\ell
\in \Lambda$, but not under $\nu\mapsto \nu + \ell_d$ with $\ell_d
\in \Lambda_d$. We therefore introduce the function
$\Phi^+_{\mu,\nu}(\tau,u,v,\{ d_r \})$ \cite{Chattopadhyaya:2025bkq},
\begin{eqnarray}
\label{eq:DefPhi+munu}
  \Phi^+_{\mu,\nu}(\tau,u,v,\{ d_r \}) &= &   \Phi_{\mu,\tilde \nu}(\tau,u,v,\{ d_r \}),
\end{eqnarray}
with
\be
\tilde \nu= \nu - \sum_{r=1}^M \lfloor \nu_r + {\rm Im}(\sigma_r) \rfloor\,d_r, 
\ee
where
\be
\sigma={\bf D}^{-1}{\bf C}^T(v-u),
\ee
with ${\bf C}$ is the $N\times M$ matrix of innerproducts
$B(\alpha_i,d_r)$ for a set of basis vectors
$\{\alpha_j\}$ of $\Lambda$, and ${\bf D}$ is the $M\times M$ matrix of
innerproducts $B(d_r,d_s)$. This function is manifestly invariant under shifts
$\nu\mapsto \nu + \ell_d$ for $\ell_d\in \Lambda_d$. 

To state the modular completion, we let $d_r^* \in \Lambda_d^*$ be
dual vectors to $d_r$ in $\Lambda_d$. We let $\{d_{s_j}\}\subsetneq \{d_r\}$ be
subsets with $M-L<N$ elements with $j=1,\dots, M-L$. Moreover, let
$\nu_g, g=1,\dots, \CN_g$ be the glue vectors for gluing the
orthogonal lattices $\Lambda_d(\{d_{s_j}\})\subsetneq \Lambda_d$ and
$\Lambda_d(\{d_r^*\}/\{d^*_{s_j}\}) \subseteq \Lambda_d$. Here
$\Lambda_d(\{d_{s_j}\})$, respectively
$\Lambda_d(\{d_r^*\}/\{d^*_{s_j}\})$, is the sublattice of $\Lambda_d$ generated by the set
$\{d_{s_j}\}$, respectively $\{d_r^*\}/\{d^*_{s_j}\}$. We can then write the
modular completion $\widehat \Phi_{\mu,\nu}$ as 
\be
\label{eq:PhiHatML2}   
\begin{split}  
  &\widehat\Phi_{\mu,\nu}(\tau,\bar \tau, u, \bar u, v,\bar v, \{ d_r \})=\Phi_{\mu,\nu }^+(\tau,u,v,\{ d_r \}) \\
&\qquad + \sum_{L=1}^{M}  \sum_{\{d_{s_j}\} \subsetneq \{d_r\}}
 \sum_{g=1}^{\CN_g} 2^{-L} R_{L,\nu_g^\perp+\nu^\perp}(\{
 d^*_r\}/\{d^*_{s_j}\},\Lambda; \tau,\bar \tau, u^\perp-v^\perp,\bar
 u^\perp -\bar v^\perp)\\
&\qquad  \times \Phi^+_{\mu,\nu_g^{||}+\nu^{||}}(\tau,u^{||},v,\{d_{s_j}\}),
\end{split}
\ee
where $u^{||}$ (respectively $u^\perp$) denotes the component of $u$
parallel (respectively orthogonal) to the hyperplane spanned by $\{d_{s_j}\}$, and the function 
$R_{L,\nu}$ is defined in terms of the generalized error
function $M_L$ by
\be
\label{eq:RLdef}
\begin{split}
R_{L,\nu}(\{d_v^*\},\Lambda; \tau,\bar \tau,u,\bar u)&=\sum_{k\in \Lambda_d(\{d_v^*\})+\nu}
M_L(\{d_v^*\},\sqrt{2\tau_2}(k-{\rm Im}(u)/\tau_2); \Lambda)\\
&\times q^{-Q(k)/2}e^{2\pi i B(u,k)}.
\end{split}
\ee 
Combined with Theorem \ref{theorem}, this provides a useful
characterization of the modular completion of the Appell
function. Using that theorem, we derive the following proposition:

\begin{prop}
The modular completion $\widehat \Phi_{\mu,\nu}$ of a depth $M$ Appell
function $\Phi_{\mu,\nu}$ is given in terms of depth $M-1$ Appell functions as
\be
\label{eq:AppellComplete}
\begin{split}
&\widehat \Phi_{\mu,\nu}(\tau,\bar
\tau, u,\bar u, v, \bar v,\{d_r\})=\Phi^+_{\mu,\nu}(\tau, u, v, \{d_r\})\\
&\quad +\frac{i}{2} \sum_{j=1}^M
\sum_{l=1}^{\CN_g^{(j)}}  \int_{-\bar \tau}^{i\infty}
dw\,\frac{\Theta_{\nu_l^{||}+\nu^{||}}(w, (u-v)_j-(\tau+w)\,{\rm Im}(u-v)_j/\tau_2;\Lambda_d(d_j^*))}{\sqrt{-i(w+\tau)}}\\
& \qquad \qquad \times \,e^{\pi i (w+\tau)
  {\rm Im}( u-v)_j^2/\tau_2^2}\, \widehat
  \Phi _{\mu,\nu_l^{||}+\nu^{||}}(\tau,-w,u^{||},\bar u^{||},v,\bar v; \{d_r\}/d_j),
\end{split}
\ee
where $u_j$ is the component of $u$ parallel to $d_j^*$,
$\CN^{(j)}_g$ the number of glue vectors $\nu_l$ gluing the lattices and
$\Lambda_d(d_j^*)$ and $\Lambda_d(\{d_{v\perp j}^*\}/d_j^*)$, and
$\Theta_\nu$ the sum over the one-dimenional sublattice
$\Lambda_d(d^*_j)\subset \Lambda_d$, 
\be
\Theta_{\nu}(\tau,u;\Lambda_d(d_j^*))=\sum_{k\in \Lambda_d(d^*_j)+\nu}
q^{Q(k)/2}\,e^{2\pi i B(u, k)}.
\ee
Moreover, $||$ in the subscripts of $\Theta$, respectively $\Phi^+$,
indicates the projection of the variable to the hyperplane spanned by
the lattice $\Lambda_d(d_j^*)$,
respectively $\Lambda_d(\{d_r\}/d_j)$
\end{prop}

\begin{proof}
Using Corollary \ref{cor:MNholanom}, it is straightforward to determine the
anti-holomorphic derivative of $R_{L,\nu}$ to $\bar \tau$. If we let
${\rm Im}(u)/\tau_2$ be independent of $\bar \tau$, this reads
\be
\begin{split}
&\partial_{\bar \tau}R_{L,\nu}(\{d_v^*\},\Lambda; \tau,\bar \tau,u,\bar
u)= \\
&\quad \frac{i}{\sqrt{2\tau_2}} \sum_{j=1}^L \sum_{l=1}^{\CN^{(j)}_g}
e^{-2\pi \frac{{\rm Im}(u_j)^2}{\tau_2}}\, \Theta_{\nu_l^{||}+\nu^{||}}(-\bar \tau, \bar u_j;\Lambda_d(d_j^*))\\
&\qquad \qquad \times
R_{L-1,\nu_l^\perp+\nu^\perp}(\{d_{v\perp j}^*\}/d_j^*,\Lambda;\tau,\bar
\tau,u^\perp,\bar u^\perp).
\end{split}
\ee
where the different ingredients are introduced in the Proposition. 
The holomorphic anomaly of the completed Appell function then becomes
\be
\begin{split}
&\partial_{\bar \tau} \widehat \Phi_{\mu,\nu}(\tau,\bar
\tau, u,\bar u, v, \bar v,\{d_r\})  = \\
& \quad \frac{i}{2\sqrt{2\tau_2}}\sum_{j=1}^M
\sum_{l=1}^{\CN_g^{(j)}} e^{-2\pi \frac{{\rm
      Im}(u-v)_j^2}{\tau_2}}\, \Theta_{\nu_l^{||}+\nu^{||}}(-\bar
\tau, \bar u_j-\bar v_j;\Lambda_d(d_j^*))\\
&\qquad \qquad \times \widehat \Phi^+_{\mu,\nu_l^{||}+\nu^{||}}(\tau,\bar
\tau,u^{||},\bar u ^{||},v,\bar v; \{d_r\}/d_j).
\end{split}
\ee
 We integrate with respect to
$\bar \tau$ to express the full modular completion $\widehat
\Phi_{\mu,\nu}$ as Eq. \eqref{eq:AppellComplete} of the Proposition.
\end{proof}

\section{Generalized error functions  $E_N$ and $M_N$}\label{sec:error}
In this section we express the error functions $M_N, E_N$
corresponding to a positive definite lattice $\Lambda$ of dimension
$N$ as iterated period integrals. Throughout this section we take
${\rm dim}(\Lambda)=N$.

\subsection{Introducing $E_P$ and $M_P$}
First we recall the definition for the generalized
error function $E_P$, $P\le N$, and the generalized complementary error function $M_P$ \cite{Alexandrov:2016enp},
\begin{equation}\label{EMp}
\begin{split}
  &E_{P}(\{c_l\},x;\Lambda) = \int_{\langle c_1,\dots ,c_{P}\rangle} \prod_{l=1}^{P} {\rm sgn}(B(c_l,y))\,e^{-\pi Q(y-x^{||})}\, d^{P}y, \\ 
& M_{P}(\{c_l\},x;\Lambda) = \sqrt{\Delta(\{c_l^{\star}\})}\left(\frac{i}{\pi}\right)^{P}\int_{\langle
  c_1,\dots ,c_{P}\rangle-ix^{||}} \prod_{l=1}^{P}
\frac{1}{B(c_l^\star,y)} \\
&\qquad \qquad \qquad \qquad \times e^{-\pi Q(y)-2\pi iB(y,x)}
d^{P}y.
\end{split}
\end{equation}
Here $x^{||}$ is the orthogonal projection of $x$ to the plane of the set
of vectors $\{c_j\}$ and $c_j^\star $ is the vector dual to $c_j$ also in
the plane spanned by the set of vectors $\{c_j\}$, ie
$B(c_i,c_j^\star)=\delta_{ij}$, $i,j=1,\dots,P,$ and $c_j^\star$ is a
linear combination of $c_i$. Finally, $\Delta(\{c_j^{\star}\})$ is the
determinant of the Gram matrix formed of the bilinear products $B(c_j^\star,c_k^\star)$ of
$c_j^\star$.\footnote{Note the subtle distinction between $*$ used in
  Eq. \eqref{eq:PhiHatML2} and $\star$ used here.} Note
 that for a given set of $P$ vectors $\{c_j\}$, the error functions $E_P$ and $M_P$ do not have an explicit
 dependence on the full $N$-dimensional vector $x\in \Lambda\otimes
 \mathbb{R}$, but only on the component $x^{||}$ parallel to
 $\{c_l\}$. Thus as function of $\{c_l\}$ and $x$, $E_P$ and $M_P$
 depend on a total of $P(P+1)$
 variables. However, $E_P$ and $M_P$ are independent of the norm of
 $\{c_l\}$, and are invariant under an $SO(P)$ rotation in the plane
 of the $\{c_l\}$. As a result, the number
 of independent arguments of $E_P$ and $M_P$ is $P(P+1)/2$.

For the evaluation of the integrals (\ref{EMp}), it is natural to parametrize the integration variables $y$ in terms of
an orthonormal basis for the hyperplane spanned by $\{c_j\}$. While the choice of basis is not unique, we will introduce
natural bases using the Gram-Schmidt orthogonalization process in
Section \ref{sec:ui} below. In this
way, we introduce a set of $P$ scalars $\{u_l\}$ which specify
$x^{||}$. With these $u_l$'s specified, we
 may omit the dependence on the lattice $\Lambda$ from the notation.

We furthermore recall that the function $E_P$ can be expressed in terms of $M_P$ as follows \cite[Prop. 3.11]{Nazaroglu:2016lmr},\cite{Alexandrov:2016enp}:
 \begin{equation}
   \label{eq:EPMP}
   \begin{split}
	E_P(\{c_l\},x;\Lambda) &=\sum_{L=0}^P \sum_{v_1,\dots
          ,v_l,w_1,\dots, w_{P-L}\in\{1,\dots,P\}}
        M_L(\{c_{v_l}\},x;\Lambda)\\
        &\quad \times \prod_{j=1}^{P-L}\sgn(B(c_{w_j}^{\perp V_L},x)),
\end{split}
      \end{equation}
where $V_L$ is the hyperplane formed of $\{c_{v_i}\}$, $i=1,\dots, L$ and $c_{w}^{\perp V_L}$ is the component of $c_w$ orthogonal to $V_L$.
This is an important ingredient for derivation of
\eqref{eq:PhiHatML2} in \cite{Chattopadhyaya:2025bkq}. In Section \ref{ortho}, we determine these functions explicitly for the $A_N$
lattice.

We may assume that the set of $P$ arguments $\{c_l\}$ splits into
two mutually orthogonal sets, $\{c_l\}=\{c_{r_i}\} \cup \{c_{s_j}\}$,
$i=1,\dots,L$ and $j=1,\dots,P-L$ elements, such that $B(c_{r_i},c_{s_j})=0$.
It then follows from the definition (\ref{EMp}) that $E_P$ and
$M_P$ factorize for such choices of $\{c_l\}$, 
\be
\label{eq:factorization}
\begin{split}
&E_P(\{c_l\},x;\Lambda)=E_{P-L}(\{c_r\},x;\Lambda)\,
E_L(\{c_s\},x;\Lambda),\\
&M_P(\{c_l\},x;\Lambda)=M_{P-L}(\{c_r\},x;\Lambda)\,
M_L(\{c_s\},x;\Lambda).\\
\end{split}  
\ee

\subsection{Construction of the scalars $\{u^{(j)}_l\}$}\label{sec:ui}

As mentioned in the previous subsection, we use the  Gram-Schmidt orthogonalization process to
introduce a set of orthonormal basis vectors $\{V_1,\dots ,V_P\}$ for
the hyperplane spanned by $\{c_1,\dots ,c_P\}$.
In the following sections, it will be important to keep track of the
starting vector $c_j$ for the process, so we add the superscript $(j)$ to $\{V^{(j)}_l\}$.
Motivated by the standard basis of the $A_N$ lattice, we set up a
Gram-Schmidt process with distinct element $c_j$, and orthonormal
bases for $\{c_1,\dots,c_{j-1}\}$, and $\{c_{j+1},\dots,c_{P}\}$. We
introduce for the component of $c_l$ orthogonal to $c_j$ the notation,
\be
\label{eq:clperpj}
  c_{l\perp j}:= c_l-\frac{B(c_l,c_j)}{Q(c_j)}c_j.
  \ee

\begin{enumerate}
\item We start from a vector $c_j$ as the first vector, which gives us
  the unit vector $V_1^{(j)}=\frac{c_j}{\sqrt{Q(c_j)}}$.
\item We consider then the vector $c_{j-1}$ for the second vector
  $V_2^{(j)}=\frac{c_{(j-1)\perp j}}{\sqrt{Q(c_{(j-1)\perp
        j})}}$. For $V_3^{(j)}$, we orthonormalize $c_{j-2}$ with
  respect to $V_1^{(j)}$ and $V_2^{(j)}$, and so on. In this way, we
  obtain the basis elements $\{V_1^{(j)},\dots, V_{j}^{(j)}\}$.
\item We then continue with the vector $c_{j+1}$ to determine the $(j+1)$-th
 basis element $V_{j+1}^{(j)}$ by orthonormalize $c_{j+1}$ with respect to
 $V_1^{(j)},\dots, V_j^{(j)}$. Proceeding in this way till $c_P$ as
 $P$-th vector, we obtain the set $\{V_{j+1}^{(j)},\dots , V_{P}^{(j)}\}$. 
\item Thus given the set $\{c_l\}$, we have constructed $P$ ordered
  sets $V^{(j)}, j=1,\dots, P$ of $P$ orthonormal vectors,
\be\label{Vj1n}
V^{(j)}= \{V_1^{(j)},\dots , V_{j}^{(j)},V_{j+1}^{(j)},\dots , V_{P}^{(j)}\},
\ee
spanning the plane of $\{c_l\}$.
\end{enumerate}

Since $M_P$ \eqref{EMp} only depends on the component $x^{||}$ of
$x\in\Lambda\otimes \mathbb{C}$, we can replace the
argument $x$ in terms of a set of $P$ scalars,
$\{u_l^{(j)}\}$, defined through the orthonormal basis $V^{(j)}$,
\begin{eqnarray}\label{uiset}
u_l^{(j)}=B(V_l^{(j)},x)=B(V_l^{(j)},x^{||}).
\end{eqnarray}
Moreover for an arbitrary point $y$ in the plane spanned by $\{c_l\}$, we introduce
coordinates $y_l^{(j)}$  with respect to basis $V_l^{(j)}$, such that  
\begin{eqnarray}
\label{yiset}
y=\sum_{l=1}^P y_l^{(j)}V_l^{(j)},\quad \forall \; j.
\end{eqnarray}

A set of scalars $u_l^{(j)}$ specifies $x^{||}$. We find it convenient
to sometimes use these as arguments of $E_P$ and
$M_P$. We therefore introduce these functions with alternative set of arguments
\be
\label{eq:scalar_not}
\begin{split}
&E_P(\{u_l^{(j)}\},\{c_l\}; \Lambda):=E_{P}(\{c_l\},x;\Lambda), \\
&M_P(\{u_l^{(j)}\},\{c_l^\star\}; \Lambda):=M_P(\{c_l\},x;\Lambda),
\end{split}
\ee
where we recall that the $\{c_l^\star\}$ are the vectors dual to
$\{c_l\}$ in the hyperplane spanned by the $\{c_l\}$. Thus only for
$P=N$, $c_l^\star$ is equivalent to $c_l^*$. For the integrals to be
well-defined, we require that $u_l^{(j)}\neq 0\, \forall \; l,j$.
\vspace{.3cm}
\\
\noindent
{\bf Remark:} Clearly not all $u^{(j)}_l$'s are independent since
these parametrize $x^{||}$. As mentioned below Eq. \eqref{EMp}, the number of independent arguments required for $M_N$ is
$N(N+1)/2$. Indeed,  Refs
\cite{Alexandrov:2016enp},\cite{Manschot:2017xcr} defined $M_2$ with
three arguments $\alpha, u_1$ and $u_2$, which are related to
the variables introduced here as
$B(c_1,c_2)/\sqrt{Q(c_1)\,Q(c_2)}=\alpha$,
$u^{(2)}_2 =u_1$, $u_1^{(2)}=u_2$, $u_1^{(1)}=\frac{u_1+\alpha
  u_2}{\sqrt{1+\alpha^2}}$ and $u_2^{(1)}=\frac{u_2-\alpha
  u_1}{\sqrt{1+\alpha^2}}$. Similarly, $M_3$ can be introduced with
six arguments \cite{Alexandrov:2019rth}.

\subsection{Iterated integrals $m_P$}
Assuming non-vanishing variables, $u_j\in \mathbb{R} \setminus 0$ for
all $j=1,\dots,P$, we  introduce
the iterated integral $m_P$, 
\be
\label{eq:mPwithouttau}
\begin{split}
m_P(u_1,\dots ,u_P)&=\left(\prod_{j=1}^P (-u_j)\right) \int_1^\infty
d\omega_1\int_{\omega_1}^\infty d\omega_2\, \dots
\int_{\omega_{P-1}}^\infty d\omega_P\\
&\quad \times \frac{e^{-\pi \sum_{l=1}^P\omega_l u_l^2}}{\sqrt{\prod_{k=1}^P \omega_k}}.
\end{split}
\ee
For our application to the completion of mock modular forms, it is convenient to
make a $\tau$-dependent change of variables from $\omega_l$ to
$w_l$ defined as,
\be
w_l=2i\tau_2\omega_l-\tau.
\ee
This expresses $m_P$ as \footnote{Note our definition
  (\ref{eq:mPwithouttau}) gives rise to the factor $i^P$. As a result,
   the definition for $m_2$ differs by a sign compared to the
  definition in \cite{Alexandrov:2016enp, Manschot:2017xcr}.}
\begin{equation}
\label{mpdef}
\begin{split}   
& m_P(u_1,\dots ,u_P) = 
 \left(\frac{i}{\sqrt{2\tau_2}}\right)^P\prod_{j=1}^P
 (u_j\,q^{u_j^2/4\tau_2})\\
 &\qquad \times \int_{-\bar\tau}^{i\infty}dw_1\int_{w_1}^{i\infty}dw_2
 \,\dots \int_{w_{P-1}}^{i\infty} dw_P
 \frac{e^{i\pi\sum_{l=1}^P
     w_lu_l^2/2\tau_2}}{\sqrt{\prod_{k=1}^P (-i(w_k+\tau))}}.
\end{split}
\end{equation} 
Note that while $\tau$ and $\bar \tau$ appear on the
right hand side of this definition, $m_P$ is independent of
$\tau$ and $\bar \tau$. This expression for $m_P$ is useful in
the modular completion of the Appell functions, since the factor
$q^{\sum_l u_l^2/4\tau_2}$ absorbs the factors of
$q^{-Q(k)/2}e^{2\pi i B(u,k)}$ in
\eqref{eq:RLdef}. This form of the generalized error functions
manifests the holomorphic anomaly and naturally occurs as path
integrals in theoretical physics \cite{Manschot:2017xcr,
  Alexandrov:2019rth, Dabholkar:2020fde}.

The following Proposition gives a variant of (\ref{mpdef}) with a generic
starting point $z\in \mathbb{H}$ for the $w_1$-integral.
\begin{prop}
The iterated integral $m_P$ satisfies
\begin{equation}
  \label{eq:mNwithz}
\begin{split}
  &m_P\left(\sqrt{\tfrac{-i(z+\tau)}{2\tau_2}}u_1,\dots ,\sqrt{\tfrac{-i(z+\tau)}{2\tau_2}}u_P\right)
  = 
 \left(\frac{i}{\sqrt{2\tau_2}}\right)^P\prod_{j=1}^P
 (u_j\,q^{u_j^2/4\tau_2})\\
 &\qquad \times \int_{z}^{i\infty}dw_1\int_{w_1}^{i\infty}dw_2\,
 \dots \int_{w_{P-1}}^{i\infty} dw_P
 \frac{e^{i\pi\sum_{l=1}^P
     w_lu_l^2/2\tau_2}}{\sqrt{\prod_{k=1}^P (-i(w_k+\tau))}}.
\end{split}
\end{equation}
\end{prop}

\begin{proof}
The change of arguments $u_l\to
\sqrt{\tfrac{-i(z+\tau)}{2\tau_2}}u_l$ on the right hand side of \eqref{eq:mPwithouttau},
together with
the change of variables $w_l=(z+\tau)\omega_l-\tau$ gives
the desired result \eqref{eq:mNwithz}. 
\end{proof}

We let ${\rm Sym}(P)$ be the set of all permutations
$\sigma=(\sigma_1,\dots, \sigma_P)\in {\rm
  Sym}(P)$ of $P$ elements, with $\sigma_l\in \{1,\dots,P\}$.
Recall that an $(R,S)$-shuffle is a permutation $\sigma \in {\rm
  Sym}(R+S)$ such that $\sigma_{1}< \dots < \sigma_R$ and
$\sigma_{R+1}<\dots < \sigma_{R+S}$. We denote the set of
$(R,S)$-shuffles by ${\rm Shuffles}(R,S)$. This set has $(R+S)!/(R!\,S!)$ elements.

Since $m_P$ is an iterated integral, multiplication is organized by
the shuffle product \cite{Chen:1977oja}:
\begin{prop}
The product of iterated integrals $m_R$ and $m_S$ satisfies
\be
\label{eq:shuffles}
\begin{split}
&  m_R(u_1,\dots,u_R)\, m_S(u_{r+1},\dots,u_{R+S})= \\
&\qquad \sum_{\sigma \in {\rm
    Shuffles}(R,S)}\, m_{R+S}(u_{\sigma_1},\dots,u_{\sigma_{R+S}}).
\end{split}
  \ee
\end{prop}

\subsection{Determination of generic $M_N$}\label{sec:Mn}
 To state Theorem \ref{theorem}, we generalize $c_{l\perp j}$ \eqref{eq:clperpj} to vectors $c_{n\perp
  m,l,\dots,k,j}$ with up to $N$ indices. To this end, we let $c_{n\perp m,l,\dots,k,j}$ be defined recursively in the number
of indices through
\be
c_{n\perp m,l,\dots,k,j}=c_{n\perp l,\dots,k,j}-\frac{B(c_{n\perp
    l,\dots,k,j}, c_{m\perp l,\dots,k,j})}{Q(c_{m\perp
      l,\dots,k,j})}\,c_{m\perp l,\dots,k,j}.
  \ee
For two indices on the left hand side, this equation is identical to
Eq. (\ref{eq:clperpj}). The vector $c_{n\perp m,l,\dots,k,j}$ is the
component of $c_n$ orthogonal to the ${\rm Span}(c_m,\dots,c_j)$,
which is independent of the order of $m,\dots, j$.

For each permutation $\sigma\in {\rm
  Sym}(P)$, an orthogonal basis of the $P$-dimensional plane spanned by
$c_{\sigma_l}$, $l=1,\dots, P$, is given by the set of vectors,
\be
\begin{split}
c_{\sigma_1}, c_{\sigma_2\perp\sigma_1},
c_{\sigma_3\perp\sigma_2\sigma_1},\dots ,c_{\sigma_P\perp\sigma_{P-1}\dots\sigma_1}.
\end{split}
\ee
This basis corresponds precisely to the basis obtained by the Gram-Schmidt process
for this permutation. For example, the basis $V^{(j)}$ \eqref{Vj1n} corresponds to the permutation
$\sigma=(j,j-1,j-2,\dots,1,j+1,\dots, P)$ after normalization.

For a set of $P$ vectors $\{c_l\}$ and $x\in \Lambda$, we define variables
  \be
  \label{eq:Defuij}
v_{j,k,\dots,m,n}=\frac{B(c_{n\perp m,l,\dots,k,j},x)}{\sqrt{Q(c_{n\perp m,l,\dots,k,j})}}.
\ee 
Note the specific order of the indices on the left and right
hand side in \eqref{eq:Defuij}. The variables $\{u^{(j)}_l\}$ \eqref{uiset} read
  in terms of these variables 
  \be
  \begin{split}
&  u^{(j)}_l=v_{j,(j-1),\dots,(j-l+1)},\qquad \quad \qquad l=1,\dots,j, \\ 
& u^{(j)}_l= v_{j, (j-1),\dots, 1,(j+1),(j+2),\dots, l},\qquad
l=j+1,\dots, N.
\end{split}
\ee

We can then state our main theorem:
\begin{thm}\label{theorem} 
Assume that the variables $v_{j,k,\dots}\in \mathbb{R}$ (\ref{eq:Defuij}) are non-vanishing for all
$j,k,\dots=1,\dots, P$. Then $M_P$ equals the sum of $P!$
iterated integrals,
\begin{equation} 
\label{MnAniter}  
\begin{split}
  &M_P(\{c_l\},x;\Lambda) =\sum_{\sigma \in {\rm Sym}(P)} m_P(v_{\sigma_1},v_{\sigma_1,\sigma_2},\dots,
    v_{\sigma_1,\sigma_2,\cdots ,\sigma_P}),
\end{split}
\end{equation}
with $m_P$ the iterated integral defined in Eq. \eqref{mpdef}.
\end{thm} 
\noindent
{\bf Remarks:}
\begin{enumerate}
\item For $N=1$, the theorem reduces to the known expression \cite{Alexandrov:2016enp},
  \be
M_1(u)=m_1(u)=\frac{i\,u}{\sqrt{2\tau_2}}\, q^{\frac{u^2}{4\tau_2}} \int_{-\bar \tau}^{i\infty} \frac{e^{\frac{\pi
    i u^2w}{2\tau_2}}}{\sqrt{-i(w+\tau)}} \,dw,
\ee
and similarly for $N=2$ \cite{Alexandrov:2016enp, bringmann2018higher}.
\item The central steps in the proof rely on a recursive structure for $M_N$
established also in \cite{Alexandrov:2016enp, bringmann2018higher} for $N=2$, and \cite[Lemma
3.5 and Proposition 3.6]{Nazaroglu:2016lmr} for general
$N$. Ref. \cite{Bringmann:2023ab} provides an expression for $E_N$ in
terms of iterated integrals in the context of false modular forms of
higher depth.
\item The right hand side of \eqref{MnAniter} involves $P!$
  terms. Due to symmetries of the lattice, the set $\{d_r\}$, and
  $(u,v,\mu,\nu)$, many terms can be be equal once the lattice sum are
  included. Sec. \ref{ANlattice} also discusses a reduction of the
  number of terms.
\end{enumerate}

\begin{proof}
We restrict the proof to the case $P=N$, such that
$\{c_l^\star\}$ is equivalent to $\{c_l^*\}$ for a set $\{c_l\}$ of
$N$ elements. The case for $P<N$
follows similarly.

  We note that similarly to $N=1,2$ \cite[Proposition
3.3, and 3.4]{Alexandrov:2016enp}, $M_N$ goes
to zero for $v_{i,j,k,\dots}\rightarrow\infty$. Away from the zeros of
all $v_{i,j,k,\dots}$, we may therefore write $M_N$ as 
\begin{eqnarray}
\label{eq:MNintlambda}
M_N(\{c_l\},x;\Lambda) &=& -\int_1^{\infty}\frac{dM_N(\{c_l\}, \lambda x; \Lambda)}{d\lambda}\, d\lambda.
\end{eqnarray}
Next we need to first understand how the integrand behaves. To do this we need to look at the different components of the integrand separately and recombine them. 
For this we use the following points
\ref{partfraction}-\ref{finalMn}, which together prove Theorem \ref{theorem}.

\subsection*{Steps of the proof of Theorem \ref{theorem}}
\begin{enumerate}
 \item 
 \noindent
{\bf Partial fraction}\label{partfraction}

Using
\be
B(x,y)=\sum_j B(c_j,x)\,B(c_j^*,y),
\ee
we have the following partial fraction,
\be
\label{partfrac}
\frac{B(x,y)}{\prod_{\ell=1}^N
  B(c_\ell^*,y)}=\sum_{j=1}^N  \frac{B(x,c_j)}{\prod_{\ell \neq j}
  B(c_\ell^*,y)}.
\ee
\item 
{\bf Separation of the orthogonal direction $c_j$}\label{sep:var}\\
Next we aim to substitute this relation in the integral \eqref{eq:MNintlambda},
and use the orthogonal basis $V^{(j)}$ in each
term. To this end, recall that $y_1^{(j)}$ is the component of $y$
along $c_j$. We can check that $B(c_{\ell}^*,y)$ is independent of
$y_1^{(j)}$ if $\ell\ne j$. This follows from writing
$$y=\sum_l V_l^{(j)}y_l^{(j)}= V_1^{(j)}y_1^{(j)}+\sum_{l=2}^N
V_l^{(j)}y_l^{(j)},$$
and noting that all $V_l^{(j)}$, $l=2,\dots,N$ are orthogonal to $c_j$.
Since $B(c_{\ell}^*,y)$ only depends on variables $y_l^{(j)}$ for
$l>1$ if $j\ne \ell$, we can do the following:
\begin{enumerate}
\item We split the $N$-dimensional integral in a $1$-dimensional
  integral over $y_1^{(j)}$, and a $(N-1)$-dimensional integral in the plane orthogonal to $c_j$;
\item All the dual vectors $c_\ell^*$, $\ell\neq j$ are orthogonal to
  $c_j$, and they also form the set of $N-1$ vectors dual to
  $c_{\ell\perp j}$, ie the components of $c_\ell$ orthogonal to
  $c_j$. Therefore, $c_{\ell}^*=(c_{\ell\perp j})^\star$ for $\ell\neq
  j$, and we can make this substitution in the product
$\prod_{\ell\ne j}B(c_{\ell}^*,y)=\prod_{\ell\ne j}B((c_{\ell\perp j})^\star,y)$.
\end{enumerate}

\vspace{0.3cm}
\item
\noindent
{\bf Integrand of $M_N$}\label{tildeMn}

We next consider the integrand of $M_N$. To this end, we define the function $\tilde M_N$ as,
\begin{eqnarray}\label{DefgenMn}
\begin{split}
   &\tilde M_N(\{c_l\},x;\Lambda)= \int  d^{N}y\,
 \frac{B(x,y)}{\prod_{l=1}^N B(c_l^*,y)}\,
 e^{-\pi Q(y)-2\pi i B(x,y)},
\end{split}
\end{eqnarray}
such that we can express $M_N$ by (\ref{eq:MNintlambda}) as,
\be\label{tilMn}
\begin{split}
&M_N(\{c_l\}, x;\Lambda)=\\
&\qquad
-2\left(\frac{i}{\pi}\right)^{N-1}\sqrt{\Delta(\{c_l^*\})}\int_1^{\infty}\frac{d\lambda}{\lambda}\,\tilde{M}_N(\{c_l\},\lambda x;\Lambda).
\end{split}
\ee
Substitution of \eqref{partfrac} in $\tilde M_N$ gives,
\begin{equation}\label{genMn1}
\begin{split}
&\tilde M_N(\{c_l\}, x,\Lambda) \\
&\qquad = \sum_{j=1}^N B(c_j,x) \int d^{N}
y\, \frac{1}{\prod_{l\ne j}B(c_{l}^*,y)}\, e^{-\pi Q(y)-2\pi i B(x,y)}.
\end{split}
\end{equation} 
Next, we use the basis $V^{(j)}$ and integrate over $y^{(j)}_1$
for each $j$. Letting $y$ in the integrand be the $(N-1)$-dimensional vector in the plane spanned by
$V^{(j)}_l$, we obtain
\be
\begin{split}
&\tilde M_N(\{c_l\}, x,\Lambda)=\sum_{j=1}^N
u^{(j)}_1\,\sqrt{Q(c_j)}\, e^{-\pi (u_1^{(j)})^2}\\
& \qquad \times \int d^{N-1}
y \, \frac{1}{\prod_{l\ne j}B(c_{l}^*,y)}\, e^{-\pi \sum_{l=2}^N (y_l^{(j)})^2-2\pi i
    u_l^{(j)}y_l^{(j)}}.
\end{split}
\ee
As explained under point (\ref{sep:var}), we can replace the vectors $c^*_l$ in the
denominator by vectors $(c_{l\perp j})^\star$ perpendicular to $c_j$. We then arrive at
\be
\label{eq:tMsumj}
\begin{split}
&\tilde M_N(\{c_l\}, x,\Lambda)=\sum_{j=1}^N
u^{(j)}_1\,\sqrt{Q(c_j)}\, e^{-\pi (u_1^{(j)})^2}\\
& \qquad \times \int d^{N-1}
y \, \frac{1}{\prod_{l\ne j}B((c_{l\perp j})^\star,y)}\, e^{-\pi \sum_{l=2}^N (y_l^{(j)})^2-2\pi i
    u_l^{(j)}y_l^{(j)}}.
\end{split}
\ee

\item {\bf Relation between determinants of Gram matrices}\label{id:minor}

\vspace{0.3cm}
\noindent
We let $\CA$ be the $N\times N$ Gram matrix of innerproducts
$B(c_k^*,c_\ell^*)$. Then the $(j,j)$ minor of $\CA$ denoted by $\text{Minor}(\CA)_{j,j}$, satisfies 
\be
\label{eq:MinorQc}
\frac{\text{Minor}(\CA)_{j,j}}{\Delta(\{c_l^*\})}=Q(c_j).
\ee
To demonstrate this identity, note that $\CA$ is the inverse matrix to the
Gram matrix of innerproducts $B(c_k,c_\ell)$. Hence by Cramer's rule, the
diagonal element $(\CA^{-1})_{j,j}=Q(c_j)$ is given by
\eqref{eq:MinorQc}. The $\text{Minor}(\CA)_{j,j}$ is independent of
the choice of basis. For us it is useful to consider the set of $N-1$
vectors $(c_{l\perp j})^\star$ dual to $c_{l\perp j}$, which are all orthogonal
to $j$. As a result, Eq. \eqref{eq:MinorQc} becomes the relation between two
determinants of Gram matrices, \be
\label{eq:RelGram}
\Delta(\{(c_{l\perp j})^{\star}\})=Q(c_j)\,\Delta(c_k^*).
\ee

\item
\noindent
{\bf Iterative form of $M_N$}\label{genericMn}

After substition of \eqref{eq:RelGram} in $\tilde M_N$ \eqref{eq:tMsumj}, it can be expressed in terms of $M_{N-1}$ \eqref{EMp},
\be
\label{eq:tMMN}
\begin{split}
\tilde M_N(\{c_l\},x;\Lambda)&= \frac{1}{\sqrt{\Delta(\{c_l^*\})}}\left(
  \frac{i}{\pi}\right)^{1-N} \\
&\quad \times \sum_{j=1}^N u_1^{(j)} e^{-\pi
  (u_1^{(j)})^2} M_{N-1}(\{c_{l\perp j}\},x;\Lambda),
\end{split}
\ee
where we define
\be
M_0=1,
\ee 
such that \eqref{eq:tMMN} agrees with $\tilde M_1$
(\ref{DefgenMn}). Substitution of \eqref{eq:tMMN} in \eqref{tilMn} 
gives us the desired iterative formula for $M_N$
\be
\label{Mnlambda}
M_N(\{c_l\},x;\Lambda)=-2 \sum_{j=1}^N \int_{1}^{\infty}d\lambda\,
u_1^{(j)}\,e^{-\pi\lambda^2 (u_1^{(j)})^2}\,M_{N-1}(\{c_{l\perp j}\},\lambda x; \Lambda).
\ee

\vspace{0.3cm}
\item
\noindent
{\bf Change of variables for $\lambda$}\label{translambda}\\
We make the change of integration variables
\be
\label{eq:chofvars}
\lambda=\sqrt{\frac{-i(z+\tau)}{2\tau_2}},
\ee
such that the integration domain for $z$ is from $-\bar\tau$ to $i\infty$.
 
\vspace{0.3cm}
\item
\noindent
{\bf Final form of $M_N$}\label{finalMn}

Finally, to connect to the form for $M_N$ as iterated integral, we use
the change of variables (\ref{eq:chofvars}) and the notation defined
in (\ref{eq:Defuij}), 
\begin{equation}\label{MnAn}
 \begin{split}
M_N(\{c_l\},x;\Lambda)& =
  \sum_{j=1}^{N}\frac{i\,u_j}{\sqrt{2\tau_2}}\,q^{\frac{u_j^2}{4\tau_2}}
    \int_{-\bar\tau}^{i\infty} \frac{dz}{\sqrt{-i(z+\tau)}}\,
    e^{\frac{i\pi z u_j^2}{2\tau_2}} \\ 
&  \times  M_{N-1}\!\left(\{c_{l\perp
    j}\},\sqrt{\tfrac{-i(z+\tau)}{2\tau_2}} \,x;\Lambda\right).
\end{split}
\end{equation}
Since the formula is recursive and we have \eqref{eq:mNwithz} for $M_1=m_1$, it follows
that $M_N$ equals the iterated integral of the Theorem.

\end{enumerate}
This concludes the proof of
 Theorem \ref{theorem}. 
\end{proof}

We conclude this section with two corollaries:
\begin{cor}
If the set of arguments $\{c_l\}$
splits into two mutually orthogonal sets $\{c_{r_i}\}$, $i=1,\dots,L$ and
$\{c_{s_j}\}$, $j=1,\dots N-L$, Theorem \ref{theorem} is compatible with the
factorization \eqref{eq:factorization}.
\end{cor}
\begin{proof}
This follows from the
definition \eqref{EMp} for the left hand side of Eq. \eqref{MnAniter}, and is of course a consequence of the proof. To argue for the
factorization from the right hand side of \eqref{MnAniter}, note that if the vector $c_l$ is orthogonal to $c_n$ for an index
$l\in \{j,k,\dots,m\}$, the variable $v_{j,k,\dots,l,\dots,m,n}$ equals
the variable $v_{j,k,\dots,m,n}$ without that index. If the set $\{c_l\}$ splits in two
mutually orthogonal sets  $\{c_{r_i}\}$ and
$\{c_{s_j}\}$, the set of variables $v_{j,k,\dots,m,n}$
split into sets with subscripts restricted to either the set $\{r_i\}$ or
$\{s_j\}$. The factorization \eqref{eq:factorization} is then a
consequence of the shuffle product for iterated integrals \eqref{eq:shuffles}.
\end{proof}

\begin{cor}
\label{cor:MNholanom}
Let the variables $v_j=u_1^{(j)}$ be as in Eqs \eqref{uiset} and \eqref{eq:Defuij},
and assume that these variables are independent of $\bar \tau$. The function
$M_N$ then satisfies the ``holomorphic anomaly equation'',
  \be
\partial_{\bar \tau}M_N(\{c_l\},\sqrt{2\tau_2}\,x;\Lambda)=
\frac{i}{\sqrt{2\tau_2}}\sum_{j=1}^N q^{v_j^2/2}\,\bar
q^{v_j^2/2}\,M_{N-1}(\{c_{l\perp j}\},\sqrt{2\tau_2}\, x;\Lambda).
\ee
\end{cor} 

\begin{proof}
This follows by straightforward differentiation.
\end{proof}

\section{Specialization to $A_N$ lattice}\label{ANlattice}
In this section, we explicitly evaluate the functions $M_N$ \eqref{MnAniter} for the
root lattice of $A_N$ ie, with the simple roots of $A_N$ as the basis
vectors. This illustrates and provides a check on the general Theorem \ref{theorem}.
While in the generic case, $M_N$ takes the form as an
integral over $M_{N-1}$ \eqref{MnAn}, we find for $A_N$ that the
integrand further factorizes \eqref{MnAn2}. Throughout this section we take $\Lambda=A_N$, $x\in\Lambda$
unless otherwise mentioned.  
  
\subsection{Construction of $\{V^{(j)}_l\}$ in $A_N$ lattice}\label{viconst}

We start by determining a set of vectors $\{V^{(j)}_l\}$ as in 
Eq. \eqref{Vj1n}. To this end, recall from \eqref{eq:RLdef} that for
the completion of an Appell function the set $\{c_l\}$ is the set of
dual vectors $\{d_l^*\}$ to the set of vectors $\{ d_l\}$. We choose the $\{d_l\}$ as
the set of $N$ simple roots of the $A_N$ lattice (or a Weyl
reflection with respect to any of the simple roots), which occurs
naturally in physical partition
functions. These correspond to the set $\{c_l^*\}$.
The inner-products of the set $\{c_l\}$ are the matrix elements of the inverse of the Cartan matrix,
\be
\label{eq:Bcicl}
B(c_j,c_l)=\min(j,l)-\frac{j\, l}{N+1},
\ee
while the vectors $c_j^*$ dual to $c_j$ are then basis elements of the $A_N$ root
lattice, such that
\begin{eqnarray}
Q(c_j^{*})=Q(d_j)=2, \quad \Delta(\{c_k^{*}\})=N+1.
\end{eqnarray} 

We consider now \eqref{MnAn} for the $A_N$ lattice. The function $M_{N-1}(\{c_{l\perp
    j}\}, x;A_N)$ depends on $\{c_{l\perp
    j}\}^\star$, which are simply the $N-1$ root vectors
  $\{d_l\}/d_j$. For $j\neq 1, N$, this set of vectors span the direct sum of two
  orthogonal lattices, $A_{j-1}\oplus A_{N-j}\subset
  A_N$.\footnote{See also the Eqs \eqref{eq:cj-lperpjrels} and
    \eqref{eq:cjpml}, which demonstrate that the $c_{(j-l)\perp j}$
    and $c_{(j+l)\perp j}$ form mutually orthogonal sets of vectors.} As a consequence, $M_{N-1}$ factorizes by \eqref{eq:factorization}.
With the notation \eqref{eq:scalar_not} for the arguments, the integral $M_N$ is therefore expressed as:
\begin{eqnarray}\label{MnAn2}
&&M_N(\{ u^{(p)}_l\},\{d_l\};A_N) =
   \sum_{j=1}^{N}\frac{i\,u_1^{(j)}}{\sqrt{2\tau_2}}q^{\frac{u_1^{(j)\,2}}{4\tau_2}}
   \int_{-\bar\tau}^{i\infty} \frac{dz}{\sqrt{-i(z+\tau)}}
   e^{\frac{i\pi z u_1^{(j)\,2}}{2\tau_2}} \nn \\ 
&& \qquad \times M_{j-1} \left(\sqrt{\tfrac{-i(z+\tau)}{2\tau_2}} \{u_2^{(j)}, \dots, u_{j}^{(j)}\},\{d_1,\dots, d_{j-1}\};A_{N}\right)\\ 
&& \qquad \times M_{N-j} \left(\sqrt{\tfrac{-i(z+\tau)}{2\tau_2}} \{u_{j+1}^{(j)},\dots, u_{N}^{(j)}\},\{d_{j+1},\dots, d_{N}\};A_{N}\right).\nn
\end{eqnarray} 
This provides an intriguing reduction of the number of iterated
integrals from the $N!$ in \eqref{MnAniter}. It is clear from
\eqref{MnAn2} that the number of terms $a_N$ for the $A_N$ lattice
satisfies the recursion for Catalan numbers, $a_N=\sum_{j=1}^N
a_{j-1}\,a_{N-j}$ with $a_0=1$. Since the growth of the Catalan 
numbers $1,1,2,5,14,\dots$ is much slower than $N!$, this leads to a reduced number of iterated
integrals for the right hand side.

\subsection{Computation of $M_3$ for $A_3$ lattice}\label{M3full}

In this section, we demonstrate the general approach of the previous
section by determining an expression of $M_3$ as an iterated period integral
depending on $u^{(j)}_l$'s. To evaluate the form of $M_3$, we let the
argument $x=\sqrt{2\tau_2}\,k \in\mathbb{R}\otimes \Lambda_d$, where
$k=k_4 d_1+k_5d_2+k_6d_3\in \Lambda_d= A_3$.  With respect to the basis of positive
simple roots $\alpha_l$, we choose vectors $d_l=c_l^*$ as in \cite[Section 4.2]{Chattopadhyaya:2025bkq}
\be
\begin{split}
d_1=\left( \begin{array}{c} -1 \\ 0 \\0 \end{array}\right),\quad d_2=\left( \begin{array}{c} 0 \\ 1 \\1 \end{array}\right),
\quad d_3=\left( \begin{array}{c} 0 \\ 0 \\ 1 \end{array}\right),
\end{split}
\ee
with dual vectors \dots. 
\be
\begin{split}
d_1^*=\left( \begin{array}{c} -\tfrac{1}{4} \\ \tfrac{1}{2} \\
               \tfrac{1}{4} \end{array}\right),\quad
           d_2^*=\left( \begin{array}{c} \tfrac{1}{2} \\ 1 \\ \tfrac{1}{2} \end{array}\right),
\quad d_3^*=\left( \begin{array}{c} \tfrac{1}{4} \\ \tfrac{1}{2} \\ \tfrac{3}{4} \end{array}\right),
\end{split}
\ee
We then determine the following $u^{(j)}_l$ \eqref{uiset}: 
\begin{enumerate}
\item For $j=1$, $u_1^{(1)},u_2^{(1)},u_3^{(1)}$ are given by
\begin{eqnarray}
&&u_1^{(1)} = \sqrt{\frac{8\tau_2}{3}}k_4, \;\; u_2^{(1)} = \sqrt{2\tau_2}\left(\sqrt{\frac{3}{2}}k_5-\sqrt{\frac{2}{3}}k_4\right),\\ \nn &&u_3^{(1)}=\sqrt{\tau_2}\left(2k_6-k_5\right).
\end{eqnarray}
\item For $j=3$,  $u_1^{(3)},u_2^{(3)},u_3^{(3)}$ are given by
\begin{eqnarray}
&&u_1^{(3)} = \sqrt{\frac{8\tau_2}{3}}k_6, \;\; u_2^{(3)} = \sqrt{2\tau_2}\left(\sqrt{\frac{3}{2}}k_5-\sqrt{\frac{2}{3}}k_6\right),\\ \nn &&u_3^{(3)}=\sqrt{\tau_2}\left(2k_4-k_5\right).
\end{eqnarray}
\item For $j=2$, $u_1^{(2)},u_2^{(2)},u_3^{(2)}$ are given by
\begin{eqnarray}
u_1^{(2)} = \sqrt{2\tau_2}k_5, \quad u_2^{(2)} &=& \sqrt{\tau_2}(2k_4-k_5),\quad  u_3^{(2)}=\sqrt{\tau_2}(2k_6-k_5).
\end{eqnarray}
\end{enumerate}
Using equation (\ref{MnAn2}), we thus evaluate $M_3$ as
\begin{align}
\label{eq:M3recur}
&M_3(\{u_l^{(p)}\},\{d_l\};A_3) = \frac{i (2k_4)}{\sqrt{3}}\,q^{\frac{2k_4^2}{3}}  \int_{-\bar\tau}^{i\infty} \frac{dz}{\sqrt{-i(z+\tau)}} \,e^{ \frac{4i\pi z k_4^2}{3}} \\ \nn
&\quad \times  M_{2} \bigg( \sqrt{-i(z+\tau)}   \left\{
                                                                                                                                                                                   \sqrt{\tfrac{3}{2}}k_5-\sqrt{\tfrac{2}{3}}k_4 ,\tfrac{1}{\sqrt{2}}(2k_6-k_5)\right\},\{d_{2}, d_{3}\};A_3\bigg)\\ \nn
&+ \frac{i (2k_6)}{\sqrt{3}}\,q^{\frac{2k_6^2}{3}}  \int_{-\bar\tau}^{i\infty} \frac{dz}{\sqrt{-i(z+\tau)}} \,e^{ \frac{4i\pi z k_6^2}{3}} \\ \nn
&\quad \times  M_{2}  \bigg(\sqrt{-i(z+\tau)}  \left\{ \sqrt{\tfrac{3}{2}}k_5-\sqrt{\tfrac{2}{3}}k_6,\tfrac{1}{\sqrt{2}}(2k_4-k_5)\right\},\{d_{2}, d_{1}\};A_3 \bigg)\\ \nn
&+ ik_5\, q^{\frac{k_5^2}{2}}  \int_{-\bar\tau}^{i\infty} \frac{dz}{\sqrt{-i(z+\tau)}} \,e^{i\pi z k_5^2} \\ \nn
&\quad \times  M_{1}\!\left(\sqrt{\tfrac{-i(z+\tau)}{2}}(2k_4-k_5)\right)M_{1}\!\left(\sqrt{\tfrac{-i(z+\tau)}{2}} (2k_6-k_5)\right). \nn
\end{align}

With the determination of the various arguments, we evaluate using
Theorem \ref{theorem} that $M_3$ can be written in this case as
\begin{equation}
\begin{split}
&  M_3(\{d^*_l\},\sqrt{2\tau_2}\,k;A_3)=m_3(\sqrt{8\tau_2/3}\,k_4,
  \sqrt{\tau_2/3}\, (3k_5-2k_4) , \sqrt{\tau_2}\,(2k_6-k_5) )\\
  &\qquad +
  m_3(\sqrt{8\tau_2/3}\,k_4,\sqrt{\tau_2/3}\,(3k_6-k_4),\sqrt{\tau_2}\,(2k_5-k_4-k_6))
  \\
  &\qquad +m_3(\sqrt{8\tau_2/3}\,k_6,
  \sqrt{\tau_2/3}\, (3k_5-2k_6) , \sqrt{\tau_2}\,(2k_4-k_5) )\\
  &\qquad +
  m_3(\sqrt{8\tau_2/3}\,k_6,\sqrt{\tau_2/3}\,(3k_4-k_6),\sqrt{\tau_2}\,(2k_5-k_4-k_6))
  \\
  & \qquad -\tfrac{i}{2}\,k_5\,(2k_4-k_5)\,(2k_6-k_5)\,
  q^{k_4^2+k_5^2+k_6^2-k_4k_5-k_5k_6}\\
  &\qquad \quad \times \int_{-\bar
    \tau}^{i\infty}dw_1 \int_{w_1}^{i\infty} dw_2 \int_{w_1}^{i\infty}
  dw_3\,\frac{e^{i \pi (w_1 k_5^2+ 2w_2 (k_4-k_5/2)^2+2w_3(k_6-k_5/2)^2) }}{\sqrt{i(w_1+\tau)(w_2+\tau)(w_3+\tau)}}.
\end{split}
\end{equation}
Note that on the last line the starting point of the $w_2$ and $w_3$
integrals is both $w_1$; this corresponds to the term in the integrand of \eqref{eq:M3recur}
which factorizes.

\section{Orthogonality relations for $E_P$ for $A_N$ lattice}\label{ortho}
In this section, we consider the function $E_N$ for an
$A_N$ lattice, and in particular the expression given in
Eq. (\ref{eq:EPMP}) in terms of complementary error functions $M_L$. A
key ingredient for (\ref{eq:EPMP}) is that the arguments of the
$\sgn$-functions involve the vectors $c_w^{\perp V_L}$, ie the components of
the vectors $c_w$ orthogonal to the set of vector $\{c_{v}\}$ entering
in the function $M_L$. For the $A_N$ lattice, this section presents an
approach to determine these vectors $c_w^{\perp V_L}$. To this end, we
derive the orthogonality relations described in \eqref{sgnrels}, which
helps to determine the decomposition of the $N$ dimensions of the $A_N$ root
          lattice into two mutually orthogonal sets. One set spanned by the $N-L$ vectors
          $\{c_{s_j}^*\}$, and one set spanned by the $L$ vectors
          $\{c_{v_l}^\star\}$. The set $\{c_{s_j}^*\}$ is directly used
          to determine the sign functions in \eqref{eq:EPMP}, while
          the set $\{c_{v_l}^\star\}$ enters in the arguments of $M_L$.

For more details on how to utilize these for finding the glue vectors in \eqref{eq:PhiHatML2}, \eqref{eq:RLdef} see \cite[Section 3]{Chattopadhyaya:2025bkq}. In Section \ref{A9} we illustrate the process for $A_9$ lattice.

\subsection{Decomposing the $E_P$'s}
We can write the modular completion in terms of $E_P$'s using the
relation between $M_L$'s and $E_P$'s \eqref{eq:EPMP}. 
We assume that $\{v_l\}$ is an ordered set with $v_l<v_{l+1}$ for all
$l$, and similarly for the set $\{s_j\}$. 
We define the following ordered sets:
\begin{eqnarray}
	S &=&\{c_{s_1},c_{s_2},\dots , c_{s_{P-L}}\},\\ \nn
	\mathcal{V} &= & \{c_{v_1},c_{v_2},\dots , c_{v_L}\}.
\end{eqnarray}

To determine the arguments of the ${\sgn}$ functions, and to construct the $M_{L}$'s relevant to the $A_N$ case we need the following orthogonality conditions.

\begin{enumerate}

\item The vector $c_j^*=2\,c_j-c_{j-1}-c_{j+1}$,
  $(j-1>0,\; j+1<N)$ is orthogonal to all $c_l$ where $l\ne j$\footnote{
For proof consider any $l \ge j+1$ and $l\le j-1$ and construct the
scalar products using the values of $B(c_l,c_j)$ \eqref{eq:Bcicl} in each case.}, which
  we have used for expressing $M_N$ as a complex integral in Appendix \ref{app:ANlattice}.

\item The following relations also hold true for $p,q\geq 1$:
\begin{subequations}\label{sgnrels}
	\begin{align}
B\!\left(c_j-\frac{N+1-j}{N+1-l}c_{l},c_{\ell}\right) &=0,\quad \forall \;\ell=1,\dots, l,\; j>l, \label{ortholeft}\\ 
B\!\left(c_j-\frac{j}{l}c_{l},c_{\ell}\right) &=0,\quad \forall \;\ell=l,\dots, N,\;\; j<l,\\ \label{orthomain1}
B\!\left(c_l-\frac{p}{p+q}c_{l+q}-\frac{q}{p+q}c_{l-p}, c_{l+q+\ell}\right) &=0,\quad \forall \;\ell=0,\dots, N-(l+q),\\ \label{orthomain2}
B\!\left(c_l-\frac{p}{p+q}c_{l+q}-\frac{q}{p+q}c_{l-p},c_{l-p-\ell}\right) &=0,\quad \forall \;\ell=0,\dots, l-p-1.
\end{align}
\end{subequations}

\end{enumerate}
For completeness we also write the relevant quadratic forms.
\begin{subequations}
	\begin{align}
	Q\!\left(c_j-\frac{N+1-j}{N+1-l}c_{l}\right) &= \frac{(N+1-j)(j-l)}{N+1-l},\quad j>l, \label{orthoQ}\\ 
	Q\!\left(c_j-\frac{j}{l}c_{l}\right) &= \frac{j(l-j)}{l},\quad j<l,\\ 
	Q\!\left(c_l-\frac{p}{p+q}c_{l+q}-\frac{q}{p+q}c_{l-p}\right) &= \frac{pq}{p+q}.
	\end{align}
\end{subequations}
These equations help to simplify the arguments of the sign functions and also to construct the $M_L$'s in
various cases.

\subsubsection*{Product of sgn-functions}
We expand $x \in \Lambda\otimes \mathbb{R}$ as
$x=\sum_{j=1}^N x_j\,d_j=\sum_{j=1}^N x_j\,c^*_j$. First, consider
that the ordered set $\mathcal{V}$ is 
such that $v_j-v_{j-1} \leq 2$. Then the ordered set $\{v_j\}$ is obtained by
removing $P-L$ non-adjacent elements from $\{1,\dots, P\}$. For the
complementary set of vectors $ S=\{c_{s_l}\}$ we have $c_{s_l+1}\ne
c_{s_{l+1}}$ for any $l$. For this specific
choice of $S$ and $\mathcal{V}$, $c_{s}^{\perp
  \mathcal{V}}=c_s^*$. The only sign functions which appear are therefore of the form:
\begin{eqnarray}\label{oneskipV}
	\prod_{j}{\rm sgn}(2\,x_{s_j}-x_{{s_j}+1}-x_{{s_j}-1}).
\end{eqnarray}
Note that here $c_{s_{j\pm 1}}\in\mathcal{V}$. We will illustrate this in the first example of the next sub-section.

More generally we let $S'\subseteq S$ be the subset of $S$ given by $\{c_{s_j}\}$ such
 that $v_{k}<s_j<v_{k+1}$ of cardinality $J=v_{k+1}-v_k-1$. Then it
 follows from the relations in Eq. \eqref{sgnrels} that the
 $\text{sgn}$ terms corresponding to $S'$ are as follows:
 \begin{eqnarray}
   \label{eq:prodSpr}
 	\prod_{c_{s_j}\in
   S'}\sgn\bigg((v_{k+1}-v_k)x_{s_j}-(v_{k+1}-s_j)x_{v_k}-(s_j-v_{k-1})
   x_{v_{k+1}}\bigg). 
 \end{eqnarray}
This result modifies when $c_{v_{k}}$ or $c_{v_{k+1}}$ are not in $\mathcal{V}$ ie, for $s_j=1,\;{\rm and/or}\; N$ for some $j$.
These can be compactly written as Eq. (\ref{eq:prodSpr}),
but if $\{s_j\}\cap \{1,N\}=1$ the corresponding $v_k,x_{k}=0$ and if $\{s_j\}\cap \{1,N\}=N$ the corresponding $v_{k+1}=N+1,x_{v_{k+1}}=0$.
All the above sign functions are obtained from \eqref{sgnrels}.

\subsubsection*{Vectors $\{c_{v_l}^\star\}$ for $M_L$'s}
Next we proceed with the vectors $\{c_{v_l}^\star\}$ for the $M_L$
functions \eqref{EMp}. Given a set of $L$ vectors $\{c_{v_l}\}$ such that $v_l<v_j$ if $l<j$, we can use the equations \eqref{sgnrels} repeatedly to obtain  $\{c_{v_l}^\star\}$. Also using \eqref{ortholeft} and \eqref{orthoQ} a set of $u_l$'s can be obtained to write the relevant $M_L$ functions.

To obtain a set of orthogonal vectors we can start by considering the
ordered set $\mathcal{V}=\{c_{v_1},c_{v_2},\dots , c_{v_L}\}$. Then using equation (\ref{ortholeft}) we write an analogous expression for the set $V^{(1)}$ in section \ref{sec:ui}, we call this new set of cardinality $L$ as $\tilde{V}^{(v_1)}$ whose $p$-th element is given by $\tilde{V}^{(v_1)}_p$ as follows:
\begin{eqnarray}
	\tilde{V}^{(v_1)}_1=c_{v_1},\quad
\tilde{V}^{(v_1)}_j=	c_{v_j}-\frac{N+1-v_j}{N+1-v_{j-1}}c_{v_{j-1}}.
\end{eqnarray} 
Using \eqref{orthoQ} we can write the orthonormal basis ${V}^{(v_1)}$ as
\begin{eqnarray}\nn
	{V}^{(v_1)}_1=\sqrt{\frac{N+1}{v_1(N+1-v_1)}}\tilde{V}^{(v_1)}_1,\quad 
{V}^{(v_1)}_j=	\sqrt{\frac{N+1-v_{j-1}}{(v_j-v_{j-1})(N+1-v_{j})}}\tilde{V}^{(v_1)}_j,\quad \\ \label{setV}
\end{eqnarray}
for $j=2,\dots, L$.
To obtain the dual vectors for the elements in $S$ we need to treat $c_{v_1},c_{v_j},c_{v_L}$ for $j\notin \{1,L\}$ separately.
\begin{eqnarray}\nn
	c_{v_1}^\star &=& \frac{1}{v_2-v_1}\bigg(\frac{v_2}{v_1}c_{v_1}-c_{v_2}\bigg),\quad c_{v_L}^\star = \frac{1}{v_L-v_{L-1}}\bigg(\frac{N+1-v_{L-1}}{N+1-v_L}c_{v_L}-c_{v_{L-1}}\bigg),\\ 
	c_{v_j}^\star &=& \frac{(v_{j+1}-v_{j-1})c_{v_j}-(v_{j+1}-v_{j})c_{v_{j-1}}-(v_{j}-v_{j-1})c_{v_{j+1}}}{(v_j-v_{j-1})(v_{j+1}-v_j)}.
\end{eqnarray}
In the following subsection we illustrate the procedure described above.

\subsection{Illustration of $A_9$}\label{A9}
In this section we will illustrate a few sample terms in $\CR_\mu$ for the $A_9$ lattice
where the Dynkin diagram is denoted by the following figure:
\vspace{.5cm}\\
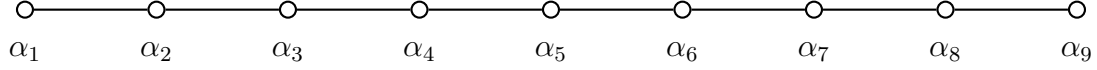
\begin{tikzpicture}[
    every node/.style={circle, draw, minimum size=6pt, inner sep=0pt},
    node distance=1.5cm,
    start chain=going right,
    thick
]
    \foreach \i in {1,...,9} {
        \node[on chain] (n\i) {};
    }
    \foreach \i in {1,...,8} {
    \draw (n\i) -- (n\the\numexpr\i+1\relax);
}

    \foreach \i in {1,...,9} {
        \node[below=5pt of n\i, draw=none, minimum size=0pt] {$\alpha_{\i}$};
    }
\end{tikzpicture}\\
We will not distinguish the basis of $A_9$ vectors as the simple roots $\{\alpha_j\}$ from the $\{c_j^*\}$. Essentially this will not change the $M_P$ or $E_P$'s appearing in the construction of $\CR_\mu$ so long as they are related to the simple roots by Weyl reflection along one of these simple roots. In the following diagrams we denote the black filled circles as the corresponding vectors which form $M_P$'s while the rest form the $\{c_{s_j}\}$.
For the argument of $\sgn()$ functions we take all $x_{j}\in\mathbb{Z}$. In general there may be a shift due to the Jacobi variable which is not relevant for the illustrative purposes of this section.

\vspace{0.3cm}
\noindent
{\bf Example 1}
\vspace{.5cm}\\
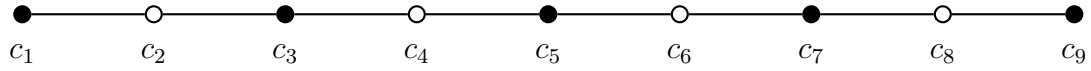
\begin{tikzpicture}[
	every node/.style={circle, draw, minimum size=6pt, inner sep=0pt},
	node distance=1.5cm,
	start chain=going right,
	thick
	]
	\foreach \i in {1,...,9} {
		\ifnum\i=1
		\node[on chain, fill=black] (n\i) {};
		\else\ifnum\i=3
		\node[on chain, fill=black] (n\i) {};
		\else\ifnum\i=5
		\node[on chain, fill=black] (n\i) {};
		\else\ifnum\i=7
		\node[on chain, fill=black] (n\i) {};
		\else\ifnum\i=9
		\node[on chain, fill=black] (n\i) {};    
		\else
		\node[on chain] (n\i) {};
		\fi\fi\fi\fi\fi
	}
	
	\foreach \i in {1,...,8} {
		\pgfmathtruncatemacro{\next}{\i+1}
		\draw (n\i) -- (n\next);
	}
	
	\foreach \i in {1,...,9} {
		\node[below=5pt of n\i, draw=none] {$c_{\i}$};
	}
	
	
\end{tikzpicture}

\noindent
The above example corresponds to the term with $$M_5(
\{c_1,c_3,c_5,c_7,c_9\},x;A_9)\times G_1.$$
In this case all $c_{s_j+1}\notin S$, so
using (\ref{oneskipV}) we find that the corresponding sign products are:
\begin{eqnarray}\nn
	&G_1=&\sgn(2x_2-x_1-x_3)\,\sgn(2x_4-x_3-x_1)\,\sgn(2x_6-x_5-x_7)\\
	&&\times\, \sgn(2x_8-x_7-x_9).
\end{eqnarray}
Using (\ref{sgnrels}) we can also write
\begin{eqnarray}
	c_1^\star &=& \frac{1}{2}\big(3c_3-c_1\big),  \quad c_3^\star = \frac{1}{2}\big(2c_3-c_1-c_5\big),\\ \nn
	c_5^\star &=& \frac{1}{2}\big(2c_5-c_3-c_7\big)  ,\quad c_7^\star = \frac{1}{2}\big(2c_7-c_5-c_9\big),\\ \nn
	c_9^\star &=& \frac{1}{2}(3c_9-c_7).
\end{eqnarray}
One set of orthogonal vectors can be given in the plane of $\{c_1,c_3,c_5,c_7,c_9\}$ using (\ref{ortholeft}) as follows:
\begin{eqnarray}
	\tilde{V}^{(v_1)}=\left\{	c_1, c_3-\frac{7}{9}c_1,c_5-\frac{5}{7}c_3,c_7-\frac{3}{5}c_5,c_9-\frac{1}{3}c_7 \right\}.
\end{eqnarray}
Using \eqref{setV} we can write the form of $M_5$ as before.

\vspace{0.3cm}
\noindent
{\bf Example 2}
\vspace{.5cm}\\
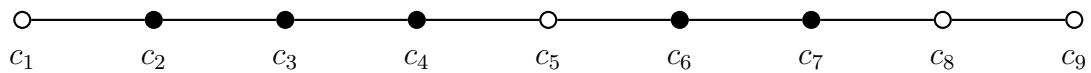
\begin{tikzpicture}[
    every node/.style={circle, draw, minimum size=6pt, inner sep=0pt},
    node distance=1.5cm,
    start chain=going right,
    thick
]
    \foreach \i in {1,...,9} {
        \ifnum\i=2
            \node[on chain, fill=black] (n\i) {};
        \else\ifnum\i=3
            \node[on chain, fill=black] (n\i) {};
        \else\ifnum\i=4
            \node[on chain, fill=black] (n\i) {};
        \else\ifnum\i=6
            \node[on chain, fill=black] (n\i) {};
        \else\ifnum\i=7
            \node[on chain, fill=black] (n\i) {};    
        \else
            \node[on chain] (n\i) {};
        \fi\fi\fi\fi\fi
    }
    
    \foreach \i in {1,...,8} {
        \pgfmathtruncatemacro{\next}{\i+1}
        \draw (n\i) -- (n\next);
    }

    \foreach \i in {1,...,9} {
        \node[below=5pt of n\i, draw=none] {$c_{\i}$};
    }


\end{tikzpicture}

\noindent
The above picture corresponds to the term with $$M_5(
\{c_2,c_3,c_4,c_6,c_7\},x;A_9)\times G_2.$$
Using the equations in (\ref{sgnrels}), we get the relevant product of $\sgn$ functions as
\begin{eqnarray}
G_2=\sgn(2x_1-x_2)\,\sgn(2x_5-x_4-x_6)\,\sgn(3x_8-2x_7)\,\sgn(3x_9-x_7).
\end{eqnarray}
Using \eqref{sgnrels} we can also write
\begin{eqnarray}
	c_2^\star &=& \frac{3}{2}c_2-c_3,\quad c_3^\star = 2c_3-c_2-c_4,\\ \nn
	c_4^\star &=& \frac{1}{2}\big(3c_4-2c_3-c_6\big),\quad c_6^\star = \frac{1}{2}\big(3c_6-c_4-2c_7\big),\\ \nn
	c_7^\star &=& \frac{4}{3}c_7-c_6.
\end{eqnarray}
An orthogonal set of vectors can be given in the plane of $\{c_2,c_3,c_4,c_6,c_7\}$ using \eqref{ortholeft} as follows:
\begin{eqnarray}
\tilde{V}^{(v_1)}=\left\{	c_2, c_3-\frac{7}{8}c_2, c_4-\frac{6}{7}c_3, c_6-\frac{2}{3}c_4, c_7-\frac{3}{4}c_6\right\}.
\end{eqnarray}

\vspace{0.3cm}
\noindent
{\bf Example 3}
\vspace{.5cm}\\
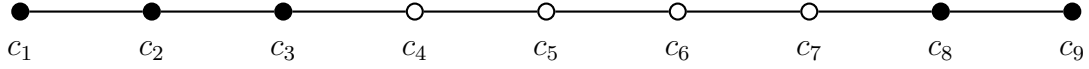
\begin{tikzpicture}[
    every node/.style={circle, draw, minimum size=6pt, inner sep=0pt},
    node distance=1.5cm,
    start chain=going right,
    thick
]
    \foreach \i in {1,...,9} {
        \ifnum\i=1
            \node[on chain, fill=black] (n\i) {};
        \else\ifnum\i=2
            \node[on chain, fill=black] (n\i) {};
        \else\ifnum\i=3
            \node[on chain, fill=black] (n\i) {};
        \else\ifnum\i=8
            \node[on chain, fill=black] (n\i) {};
        \else\ifnum\i=9
            \node[on chain, fill=black] (n\i) {};    
        \else
            \node[on chain] (n\i) {};
        \fi\fi\fi\fi\fi
    }
    
    \foreach \i in {1,...,8} {
        \pgfmathtruncatemacro{\next}{\i+1}
        \draw (n\i) -- (n\next);
    }

    \foreach \i in {1,...,9} {
        \node[below=5pt of n\i, draw=none] {$c_{\i}$};
    }


\end{tikzpicture}

\noindent
The above diagram corresponds to the term with $$M_5(
\{c_1,c_2,c_3,c_8,c_9\},x;A_9)\times G_3,$$
where $G_3$ is determined using the last three equations in
(\ref{sgnrels}) as
\begin{eqnarray}\nn
&G_3=&\sgn(5x_4-4x_3-x_8)\,\sgn(5x_5-3x_3-2x_8)\,\sgn(5x_6-2x_3-3x_8)\\ 
&&\times\, \sgn(5x_7-x_4-4x_8).
\end{eqnarray}
Similarly we can write
\begin{eqnarray}
	c_1^\star &=& 2c_1-c_2 ,\quad c_2^\star =2c_2-c_1-c_3,\\ \nn
	c_3^\star &=& \frac{1}{5}(6c_3-5c_2-c_8),\quad c_8^\star= \frac{1}{5}(6c_8-5c_9-c_3),\\ \nn
	c_9^\star &=& 2c_9-c_8.
\end{eqnarray}
Here an orthogonal set of vectors can be given in the plane of $\{c_1,c_2,c_3,c_8,c_9\}$ using \eqref{ortholeft} as follows
\begin{eqnarray}
	\tilde{V}^{(v_1)}=\left\{c_1,c_2-\frac{8}{9}c_1, c_3-\frac{7}{8}c_2, c_8-\frac{2}{7}c_3,c_9-\frac{1}{2}c_8 \right\}.
\end{eqnarray}

\vspace{0.3cm}
\noindent
{\bf Example 4}
\vspace{.5cm}\\
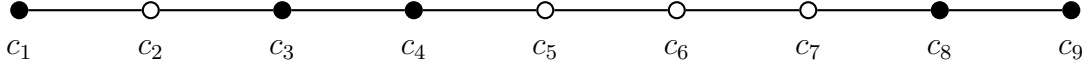
\begin{tikzpicture}[
    every node/.style={circle, draw, minimum size=6pt, inner sep=0pt},
    node distance=1.5cm,
    start chain=going right,
    thick
]
    \foreach \i in {1,...,9} {
        \ifnum\i=1
            \node[on chain, fill=black] (n\i) {};
        \else\ifnum\i=8
            \node[on chain, fill=black] (n\i) {};
        \else\ifnum\i=3
            \node[on chain, fill=black] (n\i) {};
        \else\ifnum\i=4
            \node[on chain, fill=black] (n\i) {};
        \else\ifnum\i=9
            \node[on chain, fill=black] (n\i) {};    
        \else
            \node[on chain] (n\i) {};
        \fi\fi\fi\fi\fi
    }
    
    \foreach \i in {1,...,8} {
        \pgfmathtruncatemacro{\next}{\i+1}
        \draw (n\i) -- (n\next);
    }

    \foreach \i in {1,...,9} {
        \node[below=5pt of n\i, draw=none] {$c_{\i}$};
    }

\end{tikzpicture}

\noindent
This figure corresponds to the term with $$M_5(\{
c_1,c_3,c_4,c_8,c_9\},x;A_9)
\times G_4.$$
Using \eqref{orthomain1}-\eqref{orthomain2} we get the relevant product of $\sgn$ functions as
\begin{eqnarray}\nn
G_4&=&\sgn(2x_2-x_1-x_3)\,\sgn(4x_5-3x_4-x_8)\,\sgn(4x_6-2x_4-2x_8)\\ 
&&\times \sgn(4x_7-x_4-3x_8).
\end{eqnarray}
We can also compute the relevant $M_L$ in terms of $c_{v_j}^\star,\tilde V^{(v_1)}$
\begin{eqnarray}
	c_1^\star &=& \frac{1}{2}\big(3c_1-c_3\big),\quad c_3^\star = \frac{1}{2}\big(3c_3-c_1-2c_4\big),\\ \nn
	c_4^\star &=& \frac{1}{4}\big(5c_4-4c_3-c_8\big),\quad c_8^\star = \frac{1}{4}\big(5c_8-4c_9-c_4\big),\\ \nn
	c_9^\star &=& 2c_9-c_8.
\end{eqnarray}
In this case an orthogonal set of vectors can be given in the plane of $\{c_1,c_3,c_4,c_8,c_9\}$ using \eqref{ortholeft} as follows
\begin{eqnarray}
	\tilde V^{(v_1)}=\left\{c_1,c_3-\frac{7}{9}c_1,c_4-\frac{3}{4}c_3,c_8-\frac{1}{3}c_4,c_9-\frac{1}{2}c_8\right\}.
\end{eqnarray}

\vspace{0.3cm}
\noindent
{\bf Example 5}
\vspace{.5cm}\\
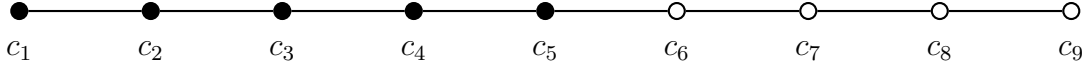
\begin{tikzpicture}[
    every node/.style={circle, draw, minimum size=6pt, inner sep=0pt},
    node distance=1.5cm,
    start chain=going right,
    thick
]
    \foreach \i in {1,...,9} {
        \ifnum\i=1
            \node[on chain, fill=black] (n\i) {};
        \else\ifnum\i=2
            \node[on chain, fill=black] (n\i) {};
        \else\ifnum\i=3
            \node[on chain, fill=black] (n\i) {};
        \else\ifnum\i=4
            \node[on chain, fill=black] (n\i) {};
        \else\ifnum\i=5
            \node[on chain, fill=black] (n\i) {};    
        \else
            \node[on chain] (n\i) {};
        \fi\fi\fi\fi\fi
    }
    
    \foreach \i in {1,...,8} {
        \pgfmathtruncatemacro{\next}{\i+1}
        \draw (n\i) -- (n\next);
    }

    \foreach \i in {1,...,9} {
        \node[below=5pt of n\i, draw=none] {$c_{\i}$};
    }

\end{tikzpicture}

\noindent
This diagram corresponds to the term with $$M_5(\{
c_1,c_2,c_3,c_4,c_5\},x;A_9)\times G_5.$$
Using the second equation in (\ref{sgnrels}) we get the relevant product of $\sgn$ functions as
\begin{eqnarray}
G_5=\sgn(5x_9-x_5)\,\sgn(5x_8-2x_5)\,\sgn(5x_7-3x_5)\sgn(5x_6-4x_5).
\end{eqnarray}
We also compute the relevant $M_L$ in terms of $c_{v_j}^\star,V^{(v_1)}$. These are given by
\begin{eqnarray}
	c_1^\star &=& 2c_1-c_2, \quad c_2^\star =2c_2-c_1-c_3,\\ \nn
	c_3^\star&=& 2c_3-c_2-c_4,\quad c_4^\star =2c_4-c_3-c_5,\\ \nn
	c_5^\star &=& \frac{6}{5}c_5-c_4.
\end{eqnarray}
We determine the orthogonal set of vectors in the plane of $\{c_1,c_2,c_3,c_4,c_5\}$ using \eqref{ortholeft}
\begin{eqnarray}
	\tilde V^{(v_1)}=\left\{c_1,c_2-\frac{8}{9}c_1,c_3-\frac{7}{8}c_2,c_4-\frac{6}{7}c_3,c_5-\frac{5}{6}c_4\right\}.
\end{eqnarray}

\appendix 
\section{Review of Vign\'eras' theorem}\label{reviewVig}
Consider an $n$-dimensional integral lattice $\Gamma$ with signature
$(n_+,n_-)$. For the purposes of this paper, we can assume that
$\Gamma$ is even, such that $Q(k)\in 2\mathbb{Z}$ for all $k\in\Gamma$.\footnote{The complete statement of Vign\'eras' theorem can be extended to lattices with a characteristic vector $p$ such that for any $k\in\Gamma$, $Q(k)+B(k,p)\in 2\mathbb{Z}$. In the present case where $\Gamma$ is even,  a zero vector is a characteristic vector.  The construction of an indefinite theta series doesn't depend on the choice of the characteristic vector in any case.}
The bilinear form is given by $B(x,x')$ where $x,x'\in\Gamma\otimes \mathbb{R}$ and $Q(x)=B(x,x)$. We have the conjugacy classes $\mu\in \Gamma^*/\Gamma$. We may construct an indefinite theta series $\vartheta_{\mu}(\tau,z)$ where $\tau\in\mathbb{H},z\in \Gamma\otimes\mathbb{C}$
where each $\vartheta_\mu$ can be written in terms of a kernel $\Phi$ as
\begin{eqnarray}
	\vartheta_\mu(\tau,z) &=& \sum_{k\in {\Gamma}+\mu} \Phi(\sqrt{2\tau_2}(k+a))\,q^{\frac{Q(k)}{2}}e^{2\pi i B(k,z)},
\end{eqnarray}
where $a={\rm Im}(z)/{\rm Im(\tau)}=\frac{z_2}{\tau_2}$ if $\tau=\tau_1+i\tau_2, z=z_1+iz_2$ and $z_1,z_2\in \Gamma \otimes \mathbb{R}$.

The kernel $\Phi(x)$ has the following two properties: 
\begin{enumerate}
	\item Take any polynomial $R(x)$ and any differential operator $D(x)$ both of degree $\le 2$. Define  $f(x)=\Phi(x)e^{\pi Q(x)/2}$. Then $\int_{-\infty}^{\infty}|f(x)|dx$ is finite or, $f(x):=L_1(\Gamma\otimes \mathbb{R})$.
	Similarly,
	$f(x),D(x)f(x), R(x)f(x)\in L_1(\Gamma\otimes \mathbb{R})\cap L_2(\Gamma\otimes \mathbb{R})$.
	\item $\Phi(x)$ also satisfies\footnote{For general kernels it is possible to have an extra term proportional to $\Phi(x)$ in \eqref{kernelprop2} which is zero in the case of $E_P,M_P$ and hence ignored for the purpose of this paper.}
	\begin{eqnarray}\label{kernelprop2}
		B^{-1}(\partial_x,\partial_x)\Phi(x)=2\pi x\partial_x\Phi(x),
	\end{eqnarray}
	where $B^{-1}$ is the bilinear form in $\Gamma^*$ and $x\partial_x$ is the Euler operator.
\end{enumerate}

 With the above kernel the indefinite theta function $\vartheta_{\mu}$ behaves like a modular form, it has the following modular transformation properties:
\begin{eqnarray}
	\vartheta_{\mu}(-\frac{1}{\tau},\frac{z}{\tau}) &=&
                                                            (-i)^{n/2}(i^{n_+})\frac{\tau^{n/2}}{\sqrt{\Gamma^*/\Gamma}}
                                                            e^{\pi i
                                                            Q(z)/\tau}\sum_{\nu
                                                            \in \Gamma^*/\Gamma} e^{-2\pi i B(\mu,\nu)}\vartheta_{\nu}(\tau,z),\\ \nn
	\vartheta_{\mu}(\tau+1,z) &=& e^{\pi i Q(\mu)}\vartheta_{\mu}(\tau,z).
\end{eqnarray}

\section{Bases for $A_N$ and derivation of $M_N$}
\label{app:ANlattice}
This section provides a detailed derivation of the error function
$M_N$ for the $A_N$ lattice in Section \ref{ANlattice}.  
We consider the set $\{c_l\}$ of vectors dual to simple roots of $A_N$
and permutations of this set. For a fixed $j$
the following $\{\tilde c_l\}$ renders an orthogonal basis, 
\begin{equation}
 \begin{split}
   &\tilde c_j=c_j, \\ 
   &\tilde c_{l}=c_{l}-\frac{l}{l+1}c_{l+1} ,\; \quad \quad \hspace{3.2cm}1\leq l<j,\\
&\tilde
c_{j+k}=c_{j+k}-\frac{N+1-(j+k)}{N+2-(j+k)}c_{j+k-1},\qquad 1\leq
k\leq N-j.  \label{tildej} 
\end{split}
\end{equation}
The quadratic forms of $\tilde c_l$ are given as
\begin{eqnarray}
Q(\tilde c_j)=\frac{j(N+1-j)}{N+1}, \quad Q(\tilde c_l)=\frac{l}{l+1}, \quad Q(\tilde c_{j+k})=\frac{N+1-j-k}{N+2-j-k}.
\end{eqnarray}
From these we can construct an orthonormal basis $\{V^{(j)}\}$ as follows:
\begin{eqnarray}
\label{eq:Vell(j)def}  
V_1^{(j)} &=& \sqrt{\frac{N+1}{(N+1-j)j}}c_j,\\ \nn
V_{\ell}^{(j)} &=& \sqrt{\frac{j-\ell+2}{j-\ell+1}}c_{j-\ell+1}-\sqrt{\frac{j-\ell+1}{j-\ell+2}}c_{j-\ell+2},\quad 1<\ell\le j,\\ \nn
V_{j+\ell}^{(j)} &=&
                     \sqrt{\frac{N+2-j-\ell}{N+1-j-\ell}}c_{j+\ell}-\sqrt{\frac{N+1-j-\ell}{N+2-j-\ell}}c_{j+\ell-1},\quad
                     1\leq \ell\le N-j.
\end{eqnarray}  

Now using $c_k^*=2c_k-(c_{k+1}+c_{k-1})$
and the above basis (\ref{tildej}), we can write the coordinates
of any vector $y\in \Lambda\otimes \mathbb{R}$ as $y=\sum_{i=1}^N y_i^{(j)}V_i^{(j)}$. For $1<j<N$, this gives $B(c_l^*,y)${\footnote{Note that $B(c_l^*,y)$ may be written in any orthogonal basis of $V^{(j)}$ hence there is dependence on $j$ at the right hand side of the equation alone.}} as follows:
\begin{eqnarray} \label{basejstar}
B(c_1^{*},y) &=& {\sqrt 2}{y^{(j)}_{j}},\\ \nn
B(c_{j-l+1}^{*},y) &=& {y^{(j)}_{l}}\sqrt{\frac{j-l+2}{j-l+1}}-{y^{(j)}_{l+1}}\sqrt{\frac{j-l}{j-l+1}}, \; {\rm for}\; l=2,\dots, j,\\ \nn
B(c_j^{*},y) &=&  \frac{y^{(j)}_1 \sqrt{N+1}}{\sqrt{j(N+1-j)}}-y^{(j)}_{j+1}\sqrt{\frac{N-j}{N-j+1}}-y^{(j)}_{2}\sqrt{\frac{j-1}{j}},\\ \nn
B(c_{j+l}^{*},y) &=& {y^{(j)}_{j+l}}\sqrt{\frac{N+2-j-l}{N+1-j-l}}-{y^{(j)}_{j+l+1}}\sqrt{\frac{N-j-l}{N+1-j-l}},\\ \nn
&& {\rm for}\;\; l=1,\dots, N-1-j,\\ \nn
B(c_N^{*},y) &=& {\sqrt 2}{y^{(j)}_N}.
\end{eqnarray}
For $j=1$, only the last three equations apply, and similarly for
$j=N$ only the first three equations apply.

Next we consider the inner products $B(c_l,y)$. Since the
$\{V^{(j)}\}$ \eqref{eq:Vell(j)def} form an orthonormal basis constructed from the
$c_j$ in a specific order, $B(c_j, V_l^{(j)})$ vanishes for $l>1$, and
similarly for other $c_l$. As a result, using
\eqref{eq:Bcicl} one finds $B(c_l,y)$ in terms of the components of
$y^{(j)}$, $j=1,\dots, N$, $k>0$, as 
\begin{align}\label{baseje}
 &B(c_j,y) = y^{(j)}_1 \sqrt{\frac{j(N+1-j)}{N+1}}, \\ \nn
 &B(c_{j-k+1},y) = y^{(j)}_1 (j-k+1) \sqrt{\frac{N+1-j}{j(N+1)}}\\ \nn
&+ \sum_{l=2}^{k}  y^{(j)}_{l}\frac{(j-k+1)}{N+1} \left((N-j+l)\sqrt{\frac{j-l+2}{j-l+1}} \right.\\ \nn
&\left. -(N-j+l-1)\sqrt{\frac{j-l+1}{j-l+2}} \right)\\ \nn
&= (j-k+1) \left(\sqrt{\frac{N+1-j}{j(N+1)}}y_1^{(j)}+\sum_{l=2}^k \frac{1}{\sqrt{(j-l+2)(j-l+1)}}y_l^{(j)}\right),\\ \nn
&B(c_{j+k},y) =  (N+1-j-k) \left( y^{(j)}_1\sqrt{\frac{j}{(N+1-j)(N+1)}}\right.\\ \nn
&+\left.\sum_{l=1}^{k}\frac{y^{(j)}_{j+l}}{\sqrt{(N+2-l-j)(N+1-l-j)}}  \right). \nn
\end{align}.

For the case $j=1$, we can solve $y_l^{(1)}$ in terms of
$B(c_k^*,y)$ from \eqref{basejstar}. The result is given by the following:
\begin{eqnarray}\label{ywithstar2}
y_N^{(1)} &=& 
              \frac{1}{\sqrt{2}}B(c_N^{*},y),  \\
\vdots \nn \\ \nn
y_{N-k+1}^{(1)} &=& \sqrt{\frac{1}{k(k+1)}}\left(k B(c_{N-k+1}^{*},y)+(k-1)B(c_{N-k+2}^{*},y)+\dots +B(c_N^{*},y)\right)\\ 
\vdots \nn \\ \nn
y_1^{(1)} &=& \sqrt{\frac{1}{N(N+1)}}\left(N B(c_1^{*},y)+(N-1)B(c_2^{*},y)+\dots +B(c_N^{*},y)\right).
\end{eqnarray}

These results were used in determining the structure for $M_N$ as a
complex integral \eqref{MnAn2}, which can also be extended to the integral
$E_N$ using the Eq. (\ref{baseje}). 
One last ingredient we need for the construction of the explicit form of $M_N$ in
Subsection \ref{viconst} is the following partial fraction. Using equation
(\ref{ywithstar2}) for the basis $(j)=(1)$, we can write,
\begin{equation}
 \begin{split}
& \qquad \sum_{j=1}^N \frac{u_j^{(1)}}{\sqrt{(N+1-j)(N+2-j)}}
\sum_{\ell=1}^{N+1-j} \frac{(N+2-j-\ell)} {\prod_{l\ne j+\ell-1}
  B(c_l^{*},y)}. \label{partialrelAn}
\end{split} 
\end{equation} 

This function $M_N$ has an interesting iterative property, which is
utilized to extract the structure of $M_N$ as an iterated period
integral \eqref{MnAniter}. Analogous to Section \ref{sec:Mn}, we can write an iterated integral structure of $\tilde M_N$, 
\begin{eqnarray}\nn
\tilde M_N(\{u_j^{(1)}\},\{c_j^*\},A_N) &=& \int \frac{\sum_j  u_j^{(1)}y^{(1)}_j}{\prod_l
                     B(c_l^{*},y^{(1)})}e^{ \sum_{j=1}^N -\pi i
                     (y_j^{(1)\,2}+2 u_j^{(1)}y_j^{(1)}) }
                     \prod_{j=1}^N dy_j^{(1)}.\\ 
\end{eqnarray}
We use (\ref{partialrelAn}) to rewrite $\tilde M_N$ as:
\begin{equation}
  \begin{split}
  \label{Mnrel0An}
& \tilde M_N(\{u_j^{(1)}\},\{c_j^*\},A_N)=\int dy_{1}^{(j+\ell-1)} e^{-\pi i y_{1}^{(j+\ell-1)\,2}-2\pi i u_{1}^{(j+\ell-1)}y_{1}^{(j+\ell-1)} }\\ 
&\qquad  \times \int \prod_{p>1} dy_p^{(j+\ell-1)}\sum_{j=1}^N
\frac{u_j^{(1)}}{\sqrt{(N+1-j)(N+2-j)}} \\ 
& \qquad \times \sum_{\ell=1}^{N+1-j}\frac{(N+2-j-\ell)} {\prod_{l\ne j+\ell-1} B(c_l^{*},y^{(j+\ell-1)})}\,e^{-\sum_{p>1} (\pi y_p^{(j+\ell-1)\,2}+2\pi i  u_p^{(j+\ell-1)} y_p^{(j+\ell-1)})} ,
\end{split}
\end{equation} 
where in each summand, we have made the change in integration
variables from the basis $(1)$ to the basis $(j+\ell-1)$. 

In the notation of Section \ref{lastpaperrev}, we note that the Schur
complement of the matrix ${\mathbf{D}}$ corresponding to the (Weyl
reflected) $A_N$ quadratic form after the projection orthogonal to the $c_j$ direction corresponds
to the quadratic form of $A_{j-1}\oplus A_{N-j}$
lattice. This can be formally checked by determining the components
of $c_{j-l}$ and $c_{j+l}$ as follows:
\begin{equation}
  \label{eq:cj-lperpjrels}
  \begin{split}
&B(c_{(j-k)\perp j},c_{(j+l) \perp j})= 0,\\ 
&B(c_{(j-k)\perp j},c_{(j-l)\perp j}) = {\rm min}(j-k,j-l)-\frac{(j-k)(j-l)}{j},\\ 
&B(c_{(j+k)\perp j},c_{(j+l)\perp j}) = {\rm
  min}(j+k,j+l)-\frac{(j+k)(j+l)}{N+j-1},
\end{split}
\end{equation}
where  
\begin{eqnarray}
\label{eq:cjpml}
c_{(j-l) \perp j}&=&c_{j-l}-\frac{j-l}{j}c_j,\\ \nn
c_{(j+k) \perp j}&=&c_{j+k}-\frac{N+1-j-k}{N+1-j}c_j.	
\end{eqnarray}
The above ensures that
\begin{enumerate}
	\item The lattice splitting occurs as described above
          \eqref{eq:cj-lperpjrels}, ie the components of $c_{j-\ell}$
          orthogonal to $c_j$ generate the lattice dual to the $A_{j}$
          root lattice.
	\item In replacing the integrand of $\tilde M_N$ with  $M_{L}$ factors we obtain the following factor $\frac{1}{\sqrt{L+1}}$ from each smaller lattice.

\end{enumerate}

This allows us to write $\tilde M_N$, using Eq. (\ref{tilMn})

\begin{align}\label{MnrelAn}
&\tilde M_N(\{u_j^{(1)}\},\{c_j^*\},A_N)\\ \nn
&\quad =\int \sum_j \frac{u_j^{(1)}y^{(1)}_j}{\prod_l B(c_l^{*},y^{(1)})}e^{ \sum_{p=1}^N -\pi i y_p^{(1)\,2}-2\pi i u_p^{(1)}y_p^{(1)} } \prod_{p=1}^N dy_p^{(1)}\\ \nn
&\quad = e^{-\pi u_{1}^{(j+\ell-1)\, 2}}   \sum_{j=1}^N \frac{u_j^{(1)}}{\sqrt{(N+1-j)(N+2-j)}} \sum_{\ell=1}^{N+1-j}\sqrt{\frac{N+2-j-\ell}{j+\ell-1}}\\ \nn
&  \qquad   \times \left(\frac{i}{\pi} \right)^{1-N} M_{j+\ell-2}(u_2^{(j+\ell-1)},\dots, u_{j+\ell-1}^{(j+\ell-1)},\{c_1^*,\dots, c_{j+\ell-2}^*\};A_{N})\\ \nn
&  \qquad   \times M_{N+1-(j+\ell)}(u_{j+\ell}^{(j+\ell-1)},\dots, u_{N}^{(j+\ell-1)},\{c_{j+\ell}^*,\dots, c_{N}^* \};A_{N}). \nn
\end{align}

To reach the final equality we observe that:
\begin{enumerate}
	\item Note that $c_j^*=d_j$ the simple roots of $A_N$ lattice. Hence, $c_j^*$ is orthogonal to all $c_{i\ne j}^* $ except $c_{j\pm 1}^*$ for any $j$. 
	\item In the context of $A_N$ lattice we have for any $\ell$ within the range $1< \ell\le j$, $V^{(j)}_{\ell}$ is orthogonal to $c_{j+p}$ where $1\le p\le N-j$. 
	\item Similarly for any $\ell$ in the range $1<\ell\le j$, $c_{\ell}$ is orthogonal to $V^{(j)}_{j+p}$ where $1\le p\le N-j$.
	\item The determinant of the matrices formed by elements of $V^{(j)}${\footnote{write $c_j$'s as column vector with entries 1 in the $j$-th position.}} for all $j$ is given by $\sqrt{N+1}$ and hence the Jacobian is $+1$ as expected from the discussion in Section \ref{sec:Mn}.
\end{enumerate}

The above result may be simplified by doing the sum on $j,\ell$. It is simpler to do the sum by fixing $j+\ell$.
Given any fixed $p=j+\ell-1$ in Eq. \eqref{MnrelAn} the contribution is given as follows:
\begin{eqnarray}\nn
&&\left(\frac{i}{\pi}\right)^{1-N}\frac{1}{\sqrt{(N+1-p)p}}\,B(c_{p},x)\,M_{p-1}(u_2^{(p)},\dots, u_{p}^{(p)},\{c_1^*,\dots, c_{p-1}^*\},A_N)\\ \nn
&& \qquad \times M_{N-p}(u_{p+1}^{(p)},\dots, u_{N}^{(p)},\{c_{p+1}^*,\dots, c_{N}^*\};A_N)\\ 
&&\quad = \left(\frac{i}{\pi}\right)^{1-N}\frac{u_1^{(p)}}{\sqrt{N+1}}M_{p-1}(u_2^{(p)},\dots, u_{p}^{(p)},\{c_1^*,\dots, c_{p-1}^*\},A_N)\\ \nn
&& \qquad \times M_{N-p}(u_{p+1}^{(p)},\dots, u_{N}^{(p)},\{c_{p+1}^*,\dots, c_{N}^*\};A_N). \label{u1punp}
\end{eqnarray}

We arrive at
\begin{align}
& M_3(\{u_{l}^{(p)}\}, \{d_l\};A_3) = -i\,q^{k_4^2+k_5^2+k_6^2-k_4k_5-k_5k_6}\int_{-\bar\tau}^{i\infty}dw_1 \int_{w_1}^{i\infty}dw_2\\ \nn &\int_{w_2}^{i\infty}\frac{dw_3}{\sqrt{i(w_1+\tau)(w_2+\tau)(w_3+\tau)}} \\ \nn
&  \bigg( \frac{k_4(2k_6-k_5)(3k_5-2k_4)}{3}   e^{\pi i \left(w_1\frac{4k_4^2}{3}+w_2\left(\frac{3k_5^2}{2}+\frac{2k_4^2}{3}-2k_4k_5\right)+w_3(2k_6^2-2k_5k_6+k_6^2/2)\right)} \\ \nn
&+  k_4(2k_5-k_4-k_6) \big(k_6-\frac{k_4}{3}\big)  e^{\pi i \left(w_1\frac{4k_4^2}{3}+w_2\left(\frac{3}{2}(k_4-k_6/3)^2 \right)+2w_3(k_5-k_4/2-k_6/2)^2\right)}  \\ \nn
&+ \frac{k_5}{2} (2k_4-k_5)(2k_6-k_5)  e^{\pi i\left(w_1 k_5^2+2w_2(k_4-k_5/2)^2+2w_3(k_6-k_5/2)^2\right)} \\ \nn
&+\{k_4\leftrightarrow k_6\}\bigg).
\end{align}

\providecommand{\href}[2]{#2}\begingroup\raggedright\endgroup


\end{document}